\documentclass[11pt,a4paper,twoside]{article}
\usepackage{amsthm,amsfonts,amsmath,amscd,amssymb}
\usepackage{latexsym}
\usepackage{euscript}
\usepackage{enumitem}
\usepackage{graphicx}
\usepackage{tikz}
\usepackage{caption}
\usepackage{subcaption}
\usepackage{cite}
\usepackage[utf8]{inputenc}
\usepackage[english]{babel}
\usepackage{xcolor}
\definecolor{darkgreen}{HTML}{00A000}
\definecolor{darkblue}{HTML}{0000A0}
\definecolor{darkred}{HTML}{D00000}
\usepackage[colorlinks=true]{hyperref}
\hypersetup{linkcolor=darkred, urlcolor=darkblue, citecolor=darkgreen}

\makeatletter 

\allowdisplaybreaks

\newcommand{\JOURNAL}{Journal of Mathematical Physics, Analysis, Geometry}

\renewcommand{\emph}[1]{\textsl{#1}}

\setlist[enumerate,1]{label=\upshape{\arabic*.}, ref=\upshape{\arabic*}, leftmargin=3.6ex, labelwidth=2ex, topsep=3pt, listparindent=3.6ex}

\addto\captionsenglish{}

\@addtoreset{section}{part}
\numberwithin{equation}{section}
\numberwithin{figure}{section}

\newcommand{\ps@jmpag}{%
\renewcommand{\@oddfoot}{}
\renewcommand{\@evenfoot}{}     
\renewcommand{\@evenhead}{\underline{\makebox[135mm][c]{\lower.5em%
       \hbox  to \textwidth{\thepage\hfil \textsl{\leftmark}  \hfil}}}}%
\renewcommand{\@oddhead}{\underline{\makebox[135mm][c]{\lower.5em%
       \hbox to \textwidth{\hfil\textsl{\rightmark} \hfil\thepage}}}}%
}%

   \renewcommand{\thebibliography}[1]{\begin{small}
    \section*{\refname}\list
   {\@biblabel{\arabic{enumiv}}}{\settowidth\labelwidth{\@biblabel{#1}}%
    \leftmargin\labelwidth
    \advance\leftmargin\labelsep
    \usecounter{enumiv}%
    \let\p@enumiv\@empty
    \def\theenumiv{\arabic{enumiv}}}%
    \def\newblock{\hskip .11em plus.33em minus.07em}%
    \sloppy\clubpenalty4000\widowpenalty4000
    \sfcode`\.=1000\relax}

\renewcommand{\endthebibliography}{%
  \def\@noitemerr{\@warning{Empty `thebibliography' environment}}%
  \endlist\end{small}
  }%

\newcommand{\key}[1]{\par\vskip 0.5em\small{\slshape Key words:\/} #1}

\newcommand{\msc}[1]{\par\vskip 0.5em\small{\slshape Mathematical Subject Classification 2010:\/} #1}

\renewenvironment{abstract}{\begin{small}\begin{quotation}\vskip -0.5em}%
                           {\end{quotation}\end{small}\par
}

\renewcommand{\@seccntformat}[1]{\csname the#1\endcsname.\quad}
\renewcommand{\section}{\@startsection{section}{1}{18pt}{3.25ex plus 1ex minus 0.2ex}%
{1.5ex plus .2ex}{\bfseries\rmfamily\large}}
\renewcommand{\subsection}{\@startsection{subsection}{2}{18pt}{3.25ex plus 1ex minus 0.2ex}%
{-1.5ex plus .2ex}{\bfseries\rmfamily}} 

\newtheoremstyle{myplain}
  {\topsep}   
  {\topsep}   
  {\itshape}  
  {18pt}       
  {\bfseries} 
  {.}         
  {5pt plus 1pt minus 1pt} 
  {}          
  
\newtheoremstyle{mydefinition}
  {\topsep}   
  {\topsep}   
  {\normalfont}  
  {18pt}       
  {\bfseries} 
  {.}         
  {5pt plus 1pt minus 1pt} 
  {}          
  
\newtheoremstyle{myremark}
  {\topsep}   
  {\topsep}   
  {\normalfont}  
  {18pt}       
  {\slshape} 
  {.}         
  {5pt plus 1pt minus 1pt} 
  {}          
\theoremstyle{myremark}

\theoremstyle{myplain}
\newtheorem{theorem}{Theorem}[section]
\newtheorem{corollary}[theorem]{Corollary}

\theoremstyle{mydefinition}
\newtheorem{definition}[theorem]{Definition}
\theoremstyle{myremark}

\newtheorem{remark}[theorem]{Remark}
\newtheorem{example}[theorem]{Example}
\renewenvironment{proof}[1][\proofname]{\par
\pushQED{\qed}%
\normalfont \topsep6\p@\@plus6\p@\relax
\trivlist
\item\relax
{\hspace{18pt}\slshape
#1\@addpunct{.}}\hspace\labelsep\ignorespaces
}{%
\popQED\endtrivlist\@endpefalse
}

\makeatother 

\begin{document}


\title{Reachability and Controllability Problems for the Heat Equation on a Half-Axis}

\author{Larissa  Fardigola and  Kateryna  Khalina}



\makeatletter 

\clearpage\thispagestyle{empty}
 \noindent
           {The final version of the paper will be published in\hfill
           \newline\JOURNAL\hfill       
           }
\medskip
\markboth{\@author}{\@title}
\refstepcounter{part}
\begin{center}%
  {\LARGE\bfseries\@title\footnote{\rule{-6mm}{4mm}\copyright\space\@author,\space 2018} \par}%
  \vskip 0.2em\null
  {\large
   \lineskip 0.5em
   \Large\@author\\[0.5em]
   \par}%
\end{center}%
 \par

\makeatother 

\markright{Reachability and Controllability Problems for the Heat Equation \ldots}



\renewcommand\textfraction{0}

\renewcommand\topfraction{1}

\renewcommand\bottomfraction{1}

\newcommand{\R}{{\mathbb R}}
\newcommand{\N}{{\mathbb N}}
\newcommand{\Z}{{\mathbb Z}}
\newcommand{\CC}{{\mathbb C}}
\newcommand{\Q}{{\mathbb Q}}
\renewcommand{\Re}{\mathop{\mathrm{Re}}}
\renewcommand{\Im}{\mathop{\mathrm{Im}}}
\newcommand{\nnl}{\left\|}
\newcommand{\nnr}{\right\|}
\newcommand{\nl}{\left|}
\newcommand{\nr}{\right|}
\newcommand{\pl}{\left(}
\newcommand{\pr}{\right)}
\newcommand{\bl}{\left[}
\newcommand{\br}{\right]}
\newcommand{\bbl}{\left\{}
\newcommand{\bbr}{\right\}}
\newcommand{\sgn}{\mathop{\mathrm{sgn}}}
\newcommand{\supp}{\mathop{\mathrm{supp}}}
\newcommand{\RR}{\EuScript{R}}
\newcommand{\FFF}{\EuScript{F}}
\newcommand{\THH}{\widetilde{H}}
\newcommand{\zco}{{\text{$\bigcirc$\kern-6.2pt\raisebox{-0.5pt}{$0$}}}}
\newcommand{\zcx}[1]{~\mbox{#1\kern-7.5pt\raisebox{0.5pt}{$\bigcirc$}}}



\begin{abstract}
In the paper, problems of controllability, approximate controllability, reachability and approximate reachability are studied for the control system $w_t=w_{xx}$, $w(0,\cdot)=u$, $x>0$, $t\in(0,T)$, where $u\in L^\infty(0,T)$ is a control. It is proved that each end state of this system is approximately reachable in a given time $T$, and each its initial state is approximately controllable in a given time $T$. A necessary and sufficient condition for reachability in a given time $T$ is obtained in terms of solvability a Markov power moment problem. It is also shown that there is no initial state that is null-controllable in a given time $T$. The results are illustrated by examples.

\key{heat equation, controllability, approximate controllability, reachability, approximate reachability, Markov power moment problem.}

\msc{93B05, 35K05, 35B30}
\end{abstract}


\section{Introduction}

Consider the  heat equation  on a half-axis
\begin{align}
&w_t=w_{xx},&& x\in(0,+\infty),\ t\in(0,T),\label{eq}\\ 
\intertext{controlled by the boundary condition}
&w(0,\cdot)=u,&& t\in(0,T),\label{bc}\\  
\intertext{under the initial condition}
&w(\cdot,0)=w^0,&&x\in(0,+\infty), \label{ic}\\
\intertext{and the seering condition}
&w(\cdot,T)=w^T,&&x\in(0,+\infty), \label{ec}
\end{align}
where $T>0$, $u\in L^\infty(0,T)$ is a control, $\pl\frac d{dt}\pr^m w:[0,T]\to H_\zco^{-2m}$, $m=0,1$, $w^0,w^T\in H_\zco^0=L^2(0,+\infty)$.  Here, for $m=0,1,2$,
\begin{align*}
H_\zco^m&=\bbl\varphi\in L^2(0,+\infty)\mid \pl \forall k=\overline{0,m}\  \varphi^{(k)}\in L^2(0,+\infty)\pr\right.\\
&\kern19.3ex\left.\wedge  \pl\forall k=\overline{0,m-1} \  \varphi^{(k)}(0^+)=0\pr\bbr
\end{align*}
with the norm
\begin{equation*}
\nnl \varphi\nnr_\zco^m=\sqrt{\sum_{k=0}^m \binom m k \pl\nnl \varphi^{(k)}\nnr_{L^2(0,+\infty)} \pr^2},
\end{equation*}
and $H_\zco^{-m}=\pl H_\zco^m\pr^*$ with the strong norm $\nnl\cdot\nnr_\zco^{-m}$ of the adjoint space. We have $H^0=L^2(0,+\infty)=\pl H_\zco^0\pr^*=H_\zco^{-0}$.

In the paper, we study  reachability and controllability problems for the heat equation on a half-axis. Note that these problems for the heat equation on domains bounded with respect to  spatial variables were investigated rather completely in a number of papers (see, e.g., \cite{MP,DE,EZua} and references therein). However controlability problems for the heat equation on domains unbounded with respect to spatial variables were not fully investigated. These problems for this equation were studied in \cite{TerZua,CMV,MZua1,CMZuaz,SMEZ}. In particular, in \cite{SMEZ}, null-controllability problem for control system \eqref{eq}--\eqref{ic} with $L^2$-control ($u\in L^2(0,T)$) was investigated in a weighted Sobolev space of negative order.  Using similarity variables and developing the solutions in the Fourier series with respect to the orthonormal basis $\{\phi_m\}_{m=1}^\infty$, the authors reduced the control problem to a moment problem
\begin{equation*}
\int_0^S e^{ms} \widetilde u(s)\, ds=\alpha_m,\quad m=\overline{1,\infty},
\end{equation*}
where $\phi_m(y)=C_m \mathcal H_{2m-1}(y/2) e^{-y^2/4}$, $\mathcal H_{2m-1}$ is the Hermit polynomial, $\alpha_m$ is determined by the Fourier coefficient of the initial state of reduced control problem, $m=\overline{1,\infty}$. The solution to the moment problem determines a solution to the control problem and vice versa. The authors proved that the moment problem admits $L^2$-solution iff $\alpha_m$ grows exponentially as $m\to \infty$. In particular, they proved that if $\alpha_m=O(e^{m\delta})$ as $m\to\infty$ for all $\delta>0$, then the initial state associated with $\{\alpha_m\}_{m=1}^\infty$ cannot be steered to the origin by $L^2$-control. In \cite{SMEZ}, it was also asserted that each initial state is approximately null-controllable in a given time $T>0$ by $L^2$-controls.

In the present paper, we study control system \eqref{eq}--\eqref{ic} in $H^0=L^2(0,+\infty)$ with $L^\infty$-control ($u\in L^\infty(0,T)$). Note that $L^\infty$-controls allow us consider initial states and solutions of the control system in the Sobolev space of order zero in contrast to \cite{SMEZ}, where the system was studied in a weighted Sobolev space of negative order as a result of using of $L^2$-controls. In Section \ref{p}, considering the odd extension with respect to $x$ of the initial state and the solution to \eqref{eq}--\eqref{ic}, we reduce this system to control system \eqref{eq1}, \eqref{ic1} in spaces $\THH^m$ of all odd functions of $H^m$. 
Further control system \eqref{eq1}, \eqref{ic1} is considered instead of control system \eqref{eq}--\eqref{ic}. In Section \ref{r}, we obtain necessary and sufficient condition for an end state $W^T$ be reachable, using controls $u\in L^\infty(0,T)$ bounded by a given constant $L>0$, from the origin. This reachability problem is reduced to an infinite Markov power moment (Theorem \ref{thmom}). Moreover, it is proved that the solutions to the finite Markov power moment problem give us control bounded by $L$ and solving the approximate reachability problem (Theorem \ref{thmomap}). 
The result of this theorem is illustrated by Examples \ref{ex1} and \ref{ex2} of  Section \ref{ex}. In Section \ref{ar}, we prove that each end state $W^T\in\THH^0$ is approximately reachable from the origin, using controls $u\in L^\infty(0,T)$, in a given time $T>0$ (Theorem \ref{thappr}). To prove this theorem, we develop $W^T$ in Fourier series with respect to $\{\psi_n^T\}_{n=0}^\infty$, $\psi_n^T(x)=\mathcal H_{2n+1}(x/\sqrt{2T}) e^{-x^2/(4T)}$, $n=\overline{0,\infty}$. 
First, for each $n=\overline{0,\infty}$, we find a sequence of controls $\{u_l^n\}_{l=0}^\infty$ that solves approximate reachability problem for the end state $\psi_n^T$. We use the Fourier transform with respect to $x$ and find these controls from the relation
\begin{equation*}
\pl\FFF \psi_n^T\pr(\sigma)
=(-1)^{n+1}i\sqrt{2T}\mathcal H_{2n+1}(\sqrt{2T}\sigma) e^{-T\sigma^2}
=-\sqrt{\frac2\pi}i\sigma \int_0^T e^{-\xi\sigma^2} u(T-\xi)\, d\xi.
\end{equation*}
Note that $u_l^n\to\delta^{(n)}$ as $l\to\infty$ in $\EuScript D'$ for each $n=\overline{0,\infty}$ ($\delta$ is the Dirac distribution).
Then we find controls $u_N$, $N\in\N$, solving the approximate reachability problem, in the form
\begin{equation*}
u_N=\sum_p^N U_p^N u_{l_p^N}^p,
\end{equation*}
where $U_p^N\ge0$ is a constant, $p=\overline{0,N}$. The results of this section are illustrated by Example \ref{ex3} of Section \ref{ex}.
In Section \ref{c}, using Theorem 3.1 of \cite{SMEZ}, we prove that there is no initial state $W^0\in\THH^0$ that is null-controllable, using controls $u\in L^\infty(0,T)$, in a given time $T>0$. 
In Section \ref{ac},  from Theorem \ref{thappr} of Section \ref{ar} it immediately follows that each initial state $W^0\in\THH^0$ is approximately controllable to any end state $W^T\in\THH^0$, using controls $u\in L^\infty(0,T)$, in a given time $T>0$.

\section{Notation}\label{n}

Introduce the spaces used in the paper. For $m=0,1,2$, denote
\begin{equation*}
H^m=\bbl\varphi\in L^2(\R)\mid \forall k=\overline{0,m}\ \varphi^{(k)}\in L^2(\R)\bbr
\end{equation*}
with the norm
\begin{equation*}
\nnl \varphi\nnr^m=\sqrt{\sum_{k=0}^m \binom m k \pl\nnl \varphi^{(k)}\nnr_{L^2(\R)} \pr^2},
\end{equation*}
and $H^{-m}=\pl H^m\pr^*$ with the strong norm $\nnl\cdot\nnr^{-m}$ of the adjoint space. We have $H^0=L^2(\R)=\pl H^0\pr^*=H^{-0}$.

For $n=\overline{-2,2}$, denote
\begin{equation*}
H_n=\bbl \psi\in L_{\text{loc}}^2(\R)\mid \pl1+\sigma^2 \pr^{n/2}\psi\in L^2(\R)\bbr
\end{equation*}
with the norm
\begin{equation*}
\nnl \psi\nnr_n=\nnl\pl1+\sigma^2 \pr^{n/2}\psi \nnr_{L^2(\R)}.
\end{equation*}
Evidently, $H_{-n}=\pl H_n\pr^*$.

By $\FFF: H^{-2}\to H_{-2}$, denote the Fourier transform operator with the domain $H^{-2}$. This operator is an extension of the classical Fourier transform operator being an isometric isomorphism of $L^2(\R)$. The extension is given by the formula
\begin{equation*}
\langle \FFF f,\varphi\rangle=\langle f,\FFF^{-1}\varphi\rangle,\quad f\in H^{-2},\ \varphi\in H_2.
\end{equation*}
This operator is an isometric isomorphism of $H^m$ and $H_m$, $m=\overline{-2,2}$ \cite[Chap. 1]{VG}.

A distribution $f\in H^{-2}$ (or $H_{-2}$) is said to be \emph{odd}  if
$\langle f,\varphi(\cdot)\rangle=-\langle f,\varphi(-(\cdot))\rangle$, $\varphi\in H^2$ (or $H_2$ respectively).

By $\THH^n$, denote the subspace of all odd distributions in $H^n$, $n=\overline{-2,2}$. Evidently, $\THH^n$ is a closed subspace of $H^n$, $n=\overline{-2,2}$.

Note that, for $\varphi \in H_\zco^m$, its odd extension $\varphi(\cdot)-\varphi(-(\cdot))$ belongs to $\THH^m$, $m=0,1,2$. But, for $m=1,2$, the converse assertion is not true. That is why the odd extension of a distribution $f\in H_\zco^{-m}$ may not belong to $\THH^{-m}$, $m=1,2$. However the following theorem holds.

\begin{theorem}[\kern-1ex\cite{LVF0}]
Let $f\in H_\zco^0$ and $f(0^+)\in\R$. Then $f''\in H_\zco^{-2}$ can be extended to the odd distribution $F$, and $F\in \THH^{-2}$. This distribution is given by the formula
\begin{equation}
\label{oddext}
F=\bigl(f(\cdot)-f(-(\cdot)) \bigr)''+2f(0^+)\delta',
\end{equation}
where $\delta$  is the Dirac distribution.
\end{theorem}

\section{Preliminary}\label{p}

Consider control problem \eqref{eq}--\eqref{ic}. Let $W^0$ and $W(\cdot,t)$ be the odd extensions of $w^0$ and $w(\cdot,t)$ with respect to $x$, $t\in [0,T]$. If $w$ is a solution to problem \eqref{eq}--\eqref{ic}, then $W$ is a solution to the following problem
\begin{align}
\label{eq1}
&W_t=W_{xx}-2u\delta',&& x\in \R, \ t\in(0,T),\\
\label{ic1}
&W(\cdot,0)=W^0, && x\in\R,
\end{align}
according to Theorem \ref{oddext}. Here  $W^0\in \THH^0$, $\pl\frac d{dt}\pr^m W:[0,T]\to\THH^{-2m}$, $m=0,1$, $\delta$ is the Dirac distribution with respect to $x$. The converse assertion is also true: if $W$ is a solution to \eqref{eq1}, \eqref{ic1}, then its restriction $w=\left.W\right|_{(0,+\infty)}$ is a solution to \eqref{eq}--\eqref{ic} and
\begin{equation}
\label{bc1}
W(0^+,t)=u(t)\quad\text{a.e. on }[0,T]
\end{equation}
(see below \eqref{soll}). Evidently, \eqref{ec} holds iff
\begin{equation}
\label{ec1}
W(\cdot,T)=W^T
\end{equation}
holds where $W^T$ is the odd extension of $w^T$.

Consider control problem \eqref{eq1}, \eqref{ic1}. Denote $V^0=\FFF W^0$ and $V(\cdot,t)=\FFF_{x\to\sigma} W(\cdot,t)$, $t\in[0,T]$. We have
\begin{align}
\label{eq2}
&V_t=-i\sigma V-\sqrt{\frac2\pi}i\sigma\,u,&& \sigma\in \R, \ t\in(0,T),\\
\label{ic2}
&V(\cdot,0)=V^0, && \sigma\in\R.
\end{align}
Therefore,
\begin{equation}
\label{sol2}
V(\sigma,t)=e^{-t\sigma^2}V^0(\sigma)-\sqrt{\frac2\pi}i\sigma \int_0^t e^{-(t-\xi)\sigma^2} u(\xi)\, d\xi,\quad\sigma\in\R,\ t\in[0,T],
\end{equation}
is the unique solution to \eqref{eq2}, \eqref{ic2}. Since $u\in L^\infty(0,T)$, we have
\begin{equation}
\label{est}
|V(\sigma,t)|\le |V^0(\sigma)|+\sqrt{\frac2\pi}\nnl u\nnr_{L^\infty(0,T)}
\frac{1-e^{-t\sigma^2}}{|\sigma|},\quad \sigma\in\R,\ t\in[0,T].
\end{equation}
Hence $V(\cdot,t)\in \THH^0$, $t\in[0,T]$. From \eqref{sol2}, we obtain
\begin{equation}
\label{sol1}
W(x,t)=\frac{e^{-\frac{x^2}{4t}}}{\sqrt{4\pi t}}*W^0(x)
+\sqrt{\frac2\pi}x\int_0^t e^{-\frac{x^2}{4\xi}}\frac{u(t-\xi)}{(2\xi)^{3/2}}d\xi.
\end{equation}
Since for any $t\in(0,T]$ the function $\frac{e^{-\frac{x^2}{4t}}}{\sqrt{2t}}*W^0(x)$ is odd and continuous, we obtain
\begin{equation*}
\frac{e^{-\frac{x^2}{4t}}}{\sqrt{2t}}*W^0(x)\to 0\quad\text{as }x\to 0^+.
\end{equation*}
Setting $\mu=\frac{|x|}{2\sqrt\xi}$, we get
\begin{equation*}
x\int_0^t e^{-\frac{x^2}{4\xi}}\frac{u(t-\xi)}{(2\xi)^{3/2}}d\xi
=\sqrt2\sgn x \int_{|x|/(2\sqrt t)}^\infty e^{-\mu^2} u\pl t-\frac{x^2}{4\mu^2}\pr d\mu.
\end{equation*}
According to Lebesgue's dominated convergence theorem, we get
\begin{equation}
\label{soll}
W(0^+,t)=\frac2{\sqrt \pi} u(t) \int_0^\infty e^{-\mu^2}=u(t)\quad\text{a.e. on }[0,T],
\end{equation}
i.e. \eqref{ec1} holds.

Thus control systems \eqref{eq}--\eqref{ic} and \eqref{eq1}, \eqref{ic1} are equivalent. That is why, further, we consider control system  \eqref{eq1}, \eqref{ic1} instead of original system \eqref{eq}--\eqref{ic}.

\section{Reachability}\label{r}

\begin{definition}
For control system \eqref{eq1}, \eqref{ic1}, a state $W^T\in\THH^0$ is said to be reachable from a state $W^0\in\THH^0$ in a given time
$T>0$ if there exists a control $u\in L^\infty(0,T)$ such that there exists a unique solution to \eqref{eq1}, \eqref{ic1}, \eqref{ec1}.
\end{definition}

By $\RR_T(W^0)$ denote the set of all states $W^T\in\THH^0$ reachable from $W^0$ in the time $T$.

According to \eqref{sol1}, we have
\begin{align}
\label{re0}
\RR_T(W^0)=&\bbl W^T\in\THH^0\mid\exists v\in L^\infty(0,T)\right.\nonumber\\
&\left.W^T=\frac{1}{\sqrt{2\pi}}\frac{e^{-\frac{x^2}{4T}}}{\sqrt{2T}}\ast W^0(x)+\sqrt{\frac{2}{\pi}}x
\int_0^T e^{-\frac{x^2}{4\xi}}\frac{v(\xi)}{(2\xi)^{3/2}}d\xi\bbr,
\end{align}
in particular,
\begin{align}
\label{re1}
\RR_T(0)=\bbl W^T\in\THH^0\mid\exists v\in L^\infty(0,T)\ \ W^T=\sqrt{\frac{2}{\pi}}x\int_0^T e^{-\frac{x^2}{4\xi}}\frac{v(\xi)}{(2\xi)^{3/2}}d\xi\bbr.
\end{align}
First, we study $\RR_T(0)$. Denote also
\begin{align}
\label{re2}
\RR_T^L(0)=&\bbl W^T\in\THH^0\mid\exists v\in L^\infty(0,T)\pl\nnl v\nnr_{L^\infty(0,T)}\leq L\right.\right.\nonumber\\
&\left.\left.\wedge W^T=\sqrt{\frac{2}{\pi}}x\int_0^T e^{-\frac{x^2}{4\xi}}\frac{v(\xi)}{(2\xi)^{3/2}}d\xi\pr\bbr.
\end{align}
Evidently, the following theorem holds
\begin{theorem}
\label{reachprop}
We have
\begin{enumerate}[label=\upshape{(\roman*)},ref=\upshape{(\roman*)}]
\item\label{laa1}
$\RR_T(0)=\cup_{L>0}\RR_T^L(0)$;
\item\label{laa2}
$\RR_T^L(0)\subset\RR_T^{L'}(0)$, $L\leq L'$;
\item\label{laa3}
$f\in\RR_T^1(0)\Leftrightarrow Lf\in\RR_T^L(0)$.
\end{enumerate}
\end{theorem}
We can obtain the following necessary condition for $f$ to belong to $\RR_T^L(0)$.
\begin{theorem}
\label{thnec}
If $W^T\in\RR_T^L(0)$, then for any $T^\ast>T$
\begin{equation}
\label{nec}
\int_0^\infty e^\frac{x^2}{4T^\ast}\nl W^T(x)\nr dx\leq L\sqrt{\frac{T^\ast}{\pi}}\ln\frac{\sqrt{T^\ast}+\sqrt{T}}{\sqrt{T^\ast}-\sqrt{T}}.
\end{equation}
\end{theorem}
\begin{proof}
Using \eqref{re2}, we have
\begin{align*}
\int_0^\infty e^\frac{x^2}{4T^\ast}\nl W^T(x)\nr dx&\leq\sqrt{\frac{2}{\pi}}L\int_0^\infty e^\frac{x^2}{4T^\ast}x
\int_0^T e^{-\frac{x^2}{4\xi}}\frac{d\xi}{(2\xi)^{3/2}}\\
&=\sqrt{\frac{2}{\pi}}L\int_0^T\frac{1}{(2\xi)^{3/2}}\int_0^\infty e^{-x^2\pl\frac{1}{4\xi}-\frac{1}{4T^\ast}\pr}xdx d\xi\\
&=\frac{L}{\sqrt{2\pi}}\int_0^T\frac{1}{(2\xi)^{3/2}}\frac{1}{\frac{1}{4\xi}-\frac{1}{4T^\ast}}d\xi
=L\sqrt{\frac{T^\ast}{\pi}}\ln\frac{\sqrt{T^\ast}+\sqrt{T}}{\sqrt{T^\ast}-\sqrt{T}}.
\end{align*}
\end{proof}
\begin{theorem}
\label{thmom}
Let $W^T\in\THH^0$ and \eqref{nec} holds. Let
\begin{equation}
\label{mom}
\omega_n=\frac{n!}{(2n+1)!}\int_0^\infty x^{2n+1}W^T(x)dx,\qquad n=\overline{0,\infty}.
\end{equation}
Then, $W^T\in\RR_T^L(0)$ iff there exists $v\in L^\infty(0,T)$ such that $\nnl v\nnr_{L^\infty(0,T)}\leq L$
and
\begin{equation}
\label{mominf}
\int_0^T \xi^n v(\xi)d\xi=\omega_n,\qquad n=\overline{0,\infty}.
\end{equation}
\end{theorem}
\begin{proof}
According to \eqref{re2}, $W^T\in\RR_T^L(0)$ iff there exists $v\in L^\infty(0,T)$
such that $\nnl v\nnr_{L^\infty(0,T)}\leq L$ and
$$
W^T=\sqrt{\frac{2}{\pi}}x\int_0^T e^{-\frac{x^2}{4\xi}}\frac{v(\xi)}{(2\xi)^{3/2}}\,d\xi.
$$
Denoting $V^T=\FFF W^T$, we have
\begin{equation*}
V^T(\sigma)=-\sqrt{\frac{2}{\pi}}i\sigma\int_0^T e^{-\xi\sigma^2}v(\xi)d\xi.
\end{equation*}
We see that $V^T(\sigma)$ is an odd entire function. Therefore,
\begin{equation*}
\sum_{n=0}^\infty\frac{\pl V^T\pr^{(2n+1)}(0)}{(2n+1)!}\sigma^{2n+1}=V^T(\sigma)
=-\sqrt{\frac{2}{\pi}}i\sigma\sum_{n=0}^\infty\frac{(-1)^n}{n!}\sigma^{2n}\int_0^T\xi^nv(\xi)d\xi.
\end{equation*}
Since
\begin{equation}
\kern-2ex\pl V^T\pr^{(2n+1)}(0)=\sqrt{\frac{2}{\pi}}\int_0^\infty(-ix)^{2n+1}W^T(x)dx
=-i\sqrt{\frac{2}{\pi}}(-1)^n\frac{(2n+1)!}{n!}\omega_n,\label{drv}
\end{equation}
we conclude the assertion of the theorem.
\end{proof}

\begin{theorem}
\label{thmomap}
Let $W^T\in\THH^0$ and \eqref{nec} holds. Let $\{\omega_n\}_{n=0}^\infty$ be defined by \eqref{mom}.
If for each $N\in\N$ there exists $v_N\in L^\infty(0,T)$ such that $\nnl v_N\nnr_{L^\infty(0,T)}\leq L$ and
\begin{equation}
\label{mpf}
\int_0^T\xi^nv_N(\xi)d\xi=\omega_n,\qquad n=\overline{0,N},
\end{equation}
then $W^T\in\overline{\RR_T^L(0)}$ (the closure is considered in $\THH^0$).
\end{theorem}
\begin{proof}
By $W_N$ denote  the solution to problem \eqref{eq1}, \eqref{ic1} with $W^0=0$ and $u(t)=v_N(T-t)$. Denote also
$V^T=\FFF W^T$, $V_N(\cdot,t)=\FFF_{x\rightarrow\sigma}W_N(\cdot,t)$, $t\in[0,T]$. Then, $V_N$ is the unique
solution to \eqref{eq2}, \eqref{ic2} with $V^0=0$ and the same $u$.
Evidently,
\begin{equation}
\label{mpf0}
\int_a^\infty \nl V^T(\sigma)\nr^2d\sigma\rightarrow 0,\quad\text{as } a\rightarrow\infty.
\end{equation}
Let $T>T^\ast$. Put
\begin{equation*}
W_{T^\ast}=\int_0^\infty e^{\frac{x^2}{4T^\ast}}\nl W^T(x)\nr dx.
\end{equation*}
For $n=\overline{0,\infty}$, we have
\begin{equation}
\pl V^T\pr^{(2n)}(0)=0,\quad  
\pl V^T\pr^{(2n+1)}(0)=(-1)^ni\sqrt{\frac{2}{\pi}}\int_0^\infty x^{2n+1}W^T(x)dx.
\end{equation}
Therefore, using the Stirling formula:
\begin{equation}
\label{stirl}
\sqrt{2\pi}n^{n+\frac{1}{2}}e^{-n}\leq n!\leq en^{n+\frac{1}{2}}e^{-n},\qquad n\in\N,
\end{equation}
we get
\begin{align}
\nl\pl V^T\pr^{(2n+1)}(0)\nr&\leq\sqrt{\frac{2}{\pi}}\int_0^\infty \pl x^{2n+1}e^{-\frac{x^2}{4T^\ast}}\pr \pl e^{\frac{x^2}{4T^\ast}}\nl W^T(x)\nr\pr dx\nonumber\\
&\leq\sqrt{\frac{2}{\pi}}W_{T^\ast}\pl\frac{2n+1}{2e}\pr^{\frac{2n+1}{2}}\pl 4T^\ast\pr^{\frac{2n+1}{2}}\nonumber\\
&
\leq W_{T^\ast}\frac{(2n+1)!}{\pi\sqrt{2n+1}}\pl\frac{2T^\ast e}{2n+1}\pr^\frac{2n+1}{2}.\label{ord3}
\end{align}
Since
\begin{equation*}
\overline{\lim_{n\rightarrow\infty}}\pl\frac{\nl\pl V^T\pr^{(2n+1)}(0)\nr}{(2n+1)!}\pr^\frac{1}{2n+1}\leq
\lim_{n\rightarrow\infty}\pl\frac{W_{T^\ast}}{\pi\sqrt{2n+1}}\pr^\frac{1}{2n+1}\sqrt{\frac{2T^\ast e}{2n+1}}=0,
\end{equation*}
we can continue $V^T$ to an odd entire function. Hence
\begin{equation}
\label{mpf1}
V^T(\sigma)=\sum_{n=0}^\infty\frac{\pl V^T\pr^{(2n+1)}(0)}{(2n+1)!}\sigma^{2n+1},\qquad \sigma\in\R.
\end{equation}
Due to \eqref{est}, we get
\begin{equation}
\label{mpf2}
\nl V_N(\sigma,T)\nr\leq\sqrt{\frac{2}{\pi}}L\frac{1-e^{-T\sigma^2}}{|\sigma|}.
\end{equation}
Hence,
\begin{align}
\int_a^\infty\nl V_N(\sigma,T)\nr^2 d\sigma &\leq\frac{2}{\pi}L^2\int_a^\infty\nl\frac{1-e^{-T\sigma^2}}{\sigma}\nr d\sigma\leq\frac{8L^2}{\pi}\int_a^\infty\frac{d\sigma}{\sigma^2}\nonumber\\
&
=\frac{8L^2}{\pi a}\to  0\quad\text{ as } a\rightarrow\infty.\label{mpf3}
\end{align}
According to \eqref{sol2}, we get
\begin{align}
\label{mpf11}
V_N(\sigma,T)&=-\sqrt{\frac{2}{\pi}}i\sigma\int_0^T e^{-\xi\sigma^2}v_N(\xi)d\xi
\nonumber\\
&=-i\sqrt{\frac{2}{\pi}}\sum_{n=0}^\infty\frac{(-1)^n}{n!}\sigma^{2n+1}\int_0^T\xi^n v_N(\xi)d\xi.
\end{align}
Due to \eqref{mpf}, we obtain
\begin{align}
V^T(\sigma)&-V_N(\sigma,T)\nonumber\\
&=\sum_{n=N+1}^\infty\sigma^{2n+1}\bl\frac{\pl V^T\pr^{(2n+1)}(0)}{(2n+1)!}-i\sqrt{\frac{2}{\pi}}\frac{(-1)^{n+1}}{n!}\int_0^T\xi^n v_N(\xi)d\xi\br.\label{mpf4}
\end{align}
With regard to \eqref{ord3} and using \eqref{stirl}, we get
\begin{align*}
\nl\frac{\pl V^T\pr^{(2n+1)}(0)}{(2n+1)!}\nr
\leq\frac{W_{T^\ast}}{\pi\sqrt{2n+1}}\pl\frac{2T^\ast e}{2n+1}\pr^\frac{2n+1}{2}
&\leq\frac{W_{T^\ast}e^{3/2}}{\pi n!\sqrt{2n+1}}\pl\frac{2T^\ast n}{2n+1}\pr^\frac{2n+1}{2}\\
&\leq\frac{W_{T^\ast}e^{3/2}}{\pi n!\sqrt{2n+1}}\pl\sqrt{T^\ast}\pr^{2n+1}.
\end{align*}
Therefore, for $|\sigma|\leq a$,
\begin{equation*}
\kern-1.5ex\nl\sum_{n=N+1}^\infty\kern-1.2ex\frac{\pl V^T\pr^{(2n+1)}(0)}{(2n+1)!}\sigma^{2n+1}\nr
\leq\frac{e^{3/2}W_{T^\ast}}{\pi}\kern-1.5ex
\sum_{n=N+1}^\infty\kern-1.2ex \frac{\pl\sqrt{T^\ast}a\pr^{2n+1}}{n!\sqrt{2n+1}}\to 0\ \text{as } N\to\infty
\end{equation*}
and
\begin{equation*}
\sqrt{\frac{2}{\pi}}\nl\sum_{n=N+1}^\infty\kern-1,2ex\frac{(-1)^{n+1}}{n!}\sigma^{2n+1}\int_0^T\xi^n v_N(\xi)d\xi\nr
\leq\sqrt{\frac{2}{\pi}}L\kern-1.2ex\sum_{n=N+1}^\infty\kern-1.2ex\frac{a^{2n+1}T^{n+1}}{(n+1)!}\to 0
\end{equation*}
as $N\to\infty$.
Taking into account \eqref{mpf4}, we get
\begin{equation*}
S_N(a)=\sup_{\sigma\in[-a,a]}\nl V^T(\sigma)-V_N(\sigma,T)\nr\to  0\quad\text{ as } N\rightarrow\infty.
\end{equation*}
Therefore,
\begin{equation}
\label{mpf7}
\int_{-a}^a\nl V^T(\sigma)-V_N(\sigma,T)\nr^2 d\sigma\leq 2a\pl S_N(a)\pr^2\to  0\quad\text{ as } N\rightarrow\infty.
\end{equation}
With regard to \eqref{mpf0}, \eqref{mpf3} and \eqref{mpf7}, we obtain
\begin{equation*}
\nnl W^T(\sigma)-W_N(\sigma,T)\nnr^0=\nnl V^T(\sigma)-V_N(\sigma,T)\nnr_0\to  0\quad\text{ as } N\rightarrow\infty,
\end{equation*}
i.e., $W^T\in\overline{\RR_T^L(0)}$.
\end{proof}


The last theorem is illustrated  by examples in Section \ref{ex} (see Examples \ref{ex1} and \ref{ex2}).

\section{Approximate reachability}\label{ar}

\begin{definition}
For control system  \eqref{eq1}, \eqref{ic1}, a state $W^T\in\THH^0$ is said to be approximately reachable  from a state $W^0\in\THH^0$
in a given time $T>0$ if $W^T\in\overline{\RR_T(W^0)}$, where the closure is considered in the space $\THH^0$.
\end{definition}

In other words, a state $W^T\in\THH^0$ is  approximately reachable from a state $W^0\in\THH^0$
in a given time $T>0$ iff for each $\varepsilon>0$ there exists $u_\varepsilon\in L^\infty(0,T)$ such that there exists a unique solution $W$ to \eqref{eq1}, \eqref{ic1} with $u=u_\varepsilon$ and $\nnl W(\cdot,T) -W^T \nnr^0<\varepsilon$.

\begin{theorem}
\label{thappr}
Each state $W^T\in\THH^0$ is approximately reachable from the origin in a given time $T>0$.
\end{theorem}

First we consider an orthogonal basis in $L^2(\R)$.
Let $\psi_n(x)=\mathcal H_n(x) e^{-\frac{x^2}{2}}$, $x\in\R$, $n=\overline{0,\infty}$, where
\begin{equation*}
\mathcal H_n(x)=(-1)^n e^{x^2}\pl\frac{d}{dx}\pr^ne^{-x^2}=n!\sum_{m=0}^{\bl\frac{n}{2}\br}\frac{(-1)^m}{m!(n-2m)!}(2x)^{n-2m}
\end{equation*}
is the Hermite polynomial, $[\cdot]$ is the integer part of a real number. It is well known \cite{HB} that
\begin{equation}
\label{ort}
\int_{-\infty}^\infty\psi_n(x)\psi_m(x)dx=\sqrt{\pi}2^nn!\delta_{mn},\qquad 0\leq m<n<+\infty,
\end{equation}
where $\delta_{mn}$ is the Kronecker delta, and $\{\psi_n\}_{n=0}^\infty$ is an orthogonal basis in $L^2(\R)$.
It is easy to see that
\begin{equation}
\label{ft}
\FFF\psi_n=(-i)^n\psi_n,\qquad n=\overline{0,\infty}.
\end{equation}
Define
\begin{align*}
\psi_n^T(x)&=\psi_{2n+1}\pl\frac{x}{\sqrt{2T}}\pr,&& x\in\R,\ n=\overline{0,\infty},\\
\widehat{\psi}_n^T(\sigma)&=\pl\FFF\psi_n^T\pr(\sigma)=(-1)^{n+1}i\sqrt{2T}\psi_{2n+1}(\sqrt{2T}\sigma),&& \sigma\in\R,\ n=\overline{0,\infty}.
\end{align*}
According to \eqref{ort},we get
\begin{equation}
\label{ortt}
\langle\psi_n^T,\psi_m^T\rangle=\langle\widehat{\psi}_n^T,\widehat{\psi}_m^T\rangle=\sqrt{2\pi T}2^{2n+1}(2n+1)!\delta_{mn},\qquad 0\leq m<n<+\infty.
\end{equation}
Obviously, $\{\psi_n^T\}_{n=0}^\infty$ and $\{\widehat{\psi}_n^T\}_{n=0}^\infty$ are orthogonal bases in $\THH^0$.
Therefore, for $f\in\THH^0$
\begin{equation*}
f=\sum_{n=0}^\infty f_n\psi_n^T,\quad\FFF f=\sum_{n=0}^\infty f_n\widehat{\psi}_n^T,\qquad\text{where }
f_n=\frac{\langle f,\psi_n^T\rangle}{\langle\psi_n^T,\psi_n^T\rangle}=\frac{\langle\FFF f,\widehat{\psi}_n^T\rangle}{\langle\psi_n^T,\psi_n^T\rangle},
\end{equation*}
and
\begin{equation}
\label{par}
\sum_{n=0}^\infty |f_n|^2\langle\widehat{\psi}_n^T,\widehat{\psi}_n^T\rangle=\sqrt{2\pi T}\sum_{n=0}^\infty |f_n|^22^{2n+1}(2n+1)!.
\end{equation}

Consider also the operator $\Phi_T:L^2(\R)\rightarrow\THH^0$ with the domain 
$D(\Phi_T)=\{g\in L^\infty(\R):\supp g\subset[0,T]\}$, acting by the rule
\begin{equation*}
\Phi_Tg=\sqrt{\frac{2}{\pi}}\FFF^{-1}\pl i\sigma\int_{-\infty}^\infty e^{-\sigma^2(T-\xi)}g(\xi)d\xi\pr,\quad g\in D(\Phi_T).
\end{equation*}
Evidently,
\begin{equation*}
\nnl\FFF\Phi_Tg\nnr_0\leq\nnl g\nnr_{L^\infty(\R)}\pl\frac{2^5T}{\pi}\pr^\frac{1}{4}.
\end{equation*}
Taking into account \eqref{sol2}, we obtain that $ W^T\in\overline{\RR_T(0)}$ iff
\begin{equation}
\label{ns}
\exists\{u_n\}_{n=1}^\infty\subset L^\infty(0,T)\quad \nnl W^T+\Phi_T u_n\nnr^0\to 0\quad\text{as } n\to\infty.
\end{equation}

Denote
\begin{align}
\kern-1.5ex&\varphi_n(\sigma)=\sigma^{2n+1}e^{-T\sigma^2},&&\kern-1ex\sigma\in\R,\nonumber\\
\kern-1.5ex&\varphi_n^l(\sigma)=\sigma^{2n+1}e^{-T\sigma^2}\pl\frac{e^{\sigma^2/l}-1}{\sigma^2/l}\pr^{n+1},
&&\kern-1ex\sigma\in\R,\nonumber\\
\kern-1.5ex&u_l^n(\xi)=\begin{cases}
(-1)^{n-j}{\binom{n}{j}} l^{n+1},&\kern-1ex \xi\in\pl\frac{j}{l},\frac{j+1}{l}\pr,\ 
j=\overline{0,n}\\
0,& \kern-1ex\xi\notin\bl 0,\frac{n+1}{l}\br
\end{cases},&&\kern-1ex l\in\N,\ n\in\N\cup\{0\}.\label{contrl}
\end{align}
Then, $\FFF\Phi_Tu_l^n=\sqrt{\frac{2}{\pi}}i\varphi_n^l$.
Figure \ref{fig:5} illustrates the  functions $u_l^n$.
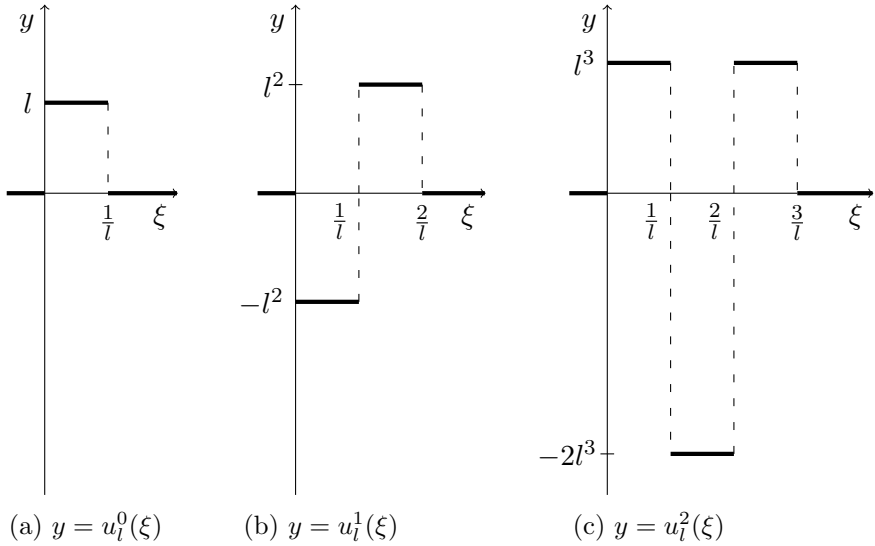
\begin{figure}[ht]
\begin{center}
\begin{subfigure}[b]{0.20\linewidth}
\centering
\begin{tikzpicture}
\draw[ultra thick] (-0.5,0) -- (0,0);
\draw (0,0) -- (0.833,0);
\draw[ultra thick] (0.833,0) node [below]{$\frac{1}{l}$} -- (1.74,0);  
\draw[->] (1.74,0) -- (1.75,0) node [below left]{$\xi$};  
\draw[->] (0,-4) -- (0,2.5) node [below left]{$y$};   
\draw[ultra thick] (0,1.2) node [left]{$l$} -- (0.833,1.2);
\draw[loosely dashed] (0.833,1.2) -- (0.833,0);
\end{tikzpicture}
\centering \caption{\parbox[t]{0.6\textwidth}{$y=u_l^0(\xi)$}}
\label{fig:5.1}
\end{subfigure}
\begin{subfigure}[b]{0.30\linewidth}
\centering
\begin{tikzpicture}
\draw[ultra thick] (-0.5,0) -- (0,0);
\draw (0,0) -- (1.666,0);
\draw[ultra thick] (1.666,0) node [below]{$\frac{2}{l}$} -- (2.49,0); 
\draw[->] (2.49,0) -- (2.5,0) node [below left]{$\xi$}; 
\draw[->] (0,-4) -- (0,2.5) node [below left]{$y$};   
\draw[ultra thick] (0,-1.44) node [left]{$-l^2$} -- (0.833,-1.44);
\draw[loosely dashed] (0.833,-1.44) -- (0.833,0) node [below left]{$\frac{1}{l}$};
\draw[loosely dashed] (0.833,0) -- (0.833,1.44);
\coordinate [label=left:$l^2$] ($l^2$) at (0,1.44);
\draw (-0.09,1.44) -- (0.09,1.44);
\draw[ultra thick] (0.833,1.44) -- (1.666,1.44);
\draw[loosely dashed] (1.666,1.44) -- (1.666,0);
\end{tikzpicture}
\centering \caption{\parbox[t]{0.6\textwidth}{$y=u_l^1(\xi)$}}
\label{fig:5.2}
\end{subfigure}
\begin{subfigure}[b]{0.35\linewidth}
\centering
\begin{tikzpicture}
\draw[ultra thick] (-0.5,0) -- (0,0);
\draw (0,0) -- (2.5,0);
\draw[ultra thick] (2.5,0) node [below]{$\frac{3}{l}$} -- (3.49,0); 
\draw[->] (3.49,0) -- (3.5,0) node [below left]{$\xi$}; 
\draw[->] (0,-4) -- (0,2.5) node [below left]{$y$}; 
\draw[ultra thick] (0,1.728) node [left]{$l^3$} -- (0.833,1.728);
\draw[loosely dashed] (0.833,1.728) -- (0.833,0) node [below left]{$\frac{1}{l}$};
\draw[loosely dashed] (0.833,0) -- (0.833,-3.456);
\coordinate [label=left:$-2l^3$] ($-2l^3$) at (0,-3.456);
\draw (-0.09,-3.456) -- (0.09,-3.456);
\draw[ultra thick] (0.833,-3.456) -- (1.666,-3.456);
\draw[loosely dashed] (1.666,-3.456) -- (1.666,0) node [below left]{$\frac{2}{l}$};
\draw[loosely dashed] (1.666,0) -- (1.666,1.728);
\draw[ultra thick] (1.666,1.728) -- (2.5,1.728);
\draw[loosely dashed] (2.5,1.728) -- (2.5,0);
\end{tikzpicture}
\centering \caption{\parbox[t]{0.6\textwidth}{$y=u_l^2(\xi)$}}
\label{fig:5.3}
\end{subfigure}
\end{center}
\centering \caption{The functions $u_l^n$.}\label{fig:5}
\end{figure}
If $l>\frac{2n+2}{T}$, we have
\begin{align*}
\nl\varphi_n^l(\sigma)\nr\leq\sigma^{2n+1}e^{-T\sigma^2}e^\frac{(n+1)\sigma^2}{l}=\sigma^{2n+1}e^{-\sigma^2\pl T-\frac{n+1}{l}\pr}
\leq\sigma^{2n+1}e^\frac{-\sigma^2T}{2}
\end{align*}
and $\varphi_n^l\to \varphi_n$ as $l\to\infty$ a.e. on $\R$. According to Lebesgue's dominated convergence theorem, we get
\begin{equation*}
\nnl\varphi_n-\varphi_n^l\nnr_0\to 0 \quad\text{ as } l\to\infty, \quad n=\overline{0,\infty}\ .
\end{equation*}

\begin{proof}[Proof of Theorem \ref{thappr}]
Let $W^T\in\THH^0$. Denote $V^T=\FFF W^T$. Then,
\begin{equation*}
W^T=\sum_{n=0}^\infty\omega_n\psi_n^T,\qquad V^T=\sum_{n=0}^\infty\omega_n\widehat{\psi}_n^T.
\end{equation*}
Due to \eqref{par}, for each $\varepsilon>0$ there exists $N\in\N$ such that
\begin{equation}
\label{esinf}
\sqrt{2\pi T}\sum_{n=N+1}^\infty |\omega_n|^22^{2n+1}(2n+1)!<\varepsilon^2.
\end{equation}
We have
\begin{equation*}
\sum_{n=0}^N\omega_n\widehat{\psi}_n^T=i\sum_{n=0}^N\omega_n\sum_{p=0}^n h_p^n\varphi_p
=i\sum_{p=0}^N\varphi_p\sum_{n=p}^N \omega_n h_p^n,
\end{equation*}
where
\begin{equation}
\label{hh}
h_p^n=\frac{(-1)^{p+1}2^{2p+1}(2T)^{p+1}}{(n-p)!(2p+1)!}(2n+1)!.
\end{equation}
For each $p=\overline{0,N}$, determine $l_p^N\in\N$ such that
\begin{equation*}
\nnl\varphi_p-\varphi_p^{l_p^N}\nnr_0<\pl\frac{\pi^3}{Te^2}\pr^\frac{1}{4}
\frac{\varepsilon }{\nnl V^T\nnr_0\sqrt{N+2}\cosh\pl 2\sqrt{2T(N+2)}\pr}
\end{equation*}
and denote
\begin{equation*}
V^T_N=i\sum_{p=0}^N\varphi_p^{l_p^N}\sum_{n=p}^N \omega_n h_p^n.
\end{equation*}
Then,
\begin{equation}
\label{gen}
\nnl V^T-V^T_N\nnr_0\leq\varepsilon\pl 1+\frac{E_N\pl\frac{\pi^3}{Te^2}\pr^{1/4}}{\nnl V^T\nnr_0\sqrt{N+2}\cosh\pl 2\sqrt{2T(N+2)}\pr}\pr,
\end{equation}
where $E_N=\sum_{p=0}^N\sum_{n=p}^N \nl\omega_n h_p^n\nr$.
Let us estimate $E_N$. For $p=\overline{0,N}$, we have
\begin{align}
\label{ee}\sum_{n=p}^N \nl\omega_n h_p^n\nr&\leq\pl\sum_{n=p}^N \nl\omega_n\nr^2\sqrt{2\pi T}2^{2n+1}(2n+1)!\pr^\frac{1}{2}
\pl\sum_{n=p}^N\frac{\nl h_p^n\nr^2}{\sqrt{2\pi T}2^{2n+1}(2n+1)!}\pr^\frac{1}{2}\nonumber\\
&\leq\nnl V^T\nnr_0\pl\sum_{n=p}^N\frac{\nl h_p^n\nr^2}{\sqrt{2\pi T}2^{2n+1}(2n+1)!}\pr^\frac{1}{2}.
\end{align}
Taking into account \eqref{hh}, we get
\begin{equation}
\label{hhh}
\frac{\nl h_p^n\nr^2}{\sqrt{2\pi T}2^{2n+1}(2n+1)!}=\frac{1}{\sqrt{2\pi T}}
\pl\frac{2^{2p+1}(2T)^{p+1}}{(2p+1)!}\pr^2\frac{(2n+1)!}{2^{2n+1}\pl(n-p)!\pr^2}.
\end{equation}
By using \eqref{stirl}, we obtain
\begin{align*}
\frac{(2n+1)!}{2^{2n+1}\pl(n-p)!\pr^2}&\leq\frac{e\sqrt{2n+1}}{2^{2n+2}\pi}\pl\frac{2n+1}{e}\pr^{2n+1}
\frac{1}{n-p}\pl\frac{e}{n-p}\pr^{2(n-p)}\\
&\leq\frac{\sqrt{2n+1}}{2\pi}\pl\frac{2n+1}{2(n-p)}\pr^{2(n-p)+1}\pl\frac{n+1}{e}\pr^{2p}.
\end{align*}
Since $\pl\frac{2n+1}{2(n-p)}\pr^{2(n-p)+1}$ is increasing with respect to $n$, we conclude that
\begin{equation*}
\sup_{n\geq p}\bbl\pl\frac{2n+1}{2(n-p)}\pr^{2(n-p)+1}\bbr=\lim_{n\rightarrow\infty}\pl\frac{2n+1}{2(n-p)}\pr^{2(n-p)+1}
=e^{2p+1}.
\end{equation*}
Therefore,
\begin{align*}
\frac{(2n+1)!}{2^{2n+1}\pl(n-p)!\pr^2}\leq\frac{\sqrt{2n+1}}{2\pi}e^{2p+1}\pl\frac{n+1}{e}\pr^{2p}
\leq\frac{e}{\sqrt{2}\pi}(n+1)^{2p+\frac{1}{2}}.
\end{align*}
According to \eqref{hhh}, we get
\begin{equation*}
\frac{\nl h_p^n\nr^2}{\sqrt{2\pi T}2^{2n+1}(2n+1)!}
\leq\frac{1}{\sqrt{2\pi T}}\pl\frac{2^{2p+1}(2T)^{p+1}}{(2p+1)!}\pr^2\frac{e}{\sqrt{2}\pi}(n+1)^{2p+\frac{1}{2}}.
\end{equation*}
Taking into account \eqref{ee}, we have
\begin{align}
\label{eh}\sum_{n=p}^N \nl\omega_n h_p^n\nr\leq\nnl V^T\nnr_0\pl\frac{1}{4\pi T}\pr^{1/4}\sqrt{\frac{e}{\pi}}
\frac{2^{2p+1}(2T)^{p+1}}{(2p+1)!}\pl\sum_{n=p}^N (n+1)^{2p+\frac{1}{2}}\pr^\frac{1}{2}.
\end{align}
Since
\begin{equation*}
\sum_{n=p}^N(n+1)^{2p+\frac{1}{2}}\leq\int_p^{N+1}(x+1)^{2p+\frac{1}{2}}dx,
\end{equation*}
we obtain
\begin{align*}
\sum_{n=p}^N \nl\omega_n h_p^n\nr&\leq\nnl V^T\nnr_0\pl\frac{1}{4\pi T}\pr^\frac{1}{4}\sqrt{\frac{e}{\pi}}
\frac{2^{2p+1}(2T)^{p+1}}{(2p+1)!}(N+2)^{p+1}\\
&=\nnl V^T\nnr_0\pl\frac{Te^2}{\pi^3}\pr^\frac{1}{4}\frac{\pl 2\sqrt{2T(N+2)}\pr^{2p+1}}{(2p+1)!}\sqrt{N+2}.
\end{align*}
Hence,
\begin{align*}
E_N&\leq\nnl V^T\nnr_0\pl\frac{Te^2}{\pi^3}\pr^\frac{1}{4}\sqrt{N+2}\sum_{p=0}^N\frac{\pl 2\sqrt{2T(N+2)}\pr^{2p+1}}{(2p+1)!}\\
&=\nnl V^T\nnr_0\pl\frac{Te^2}{\pi^3}\pr^\frac{1}{4}\sqrt{N+2}\cosh\pl 2\sqrt{2T(N+2)}\pr.
\end{align*}
Taking into account \eqref{gen}, we conclude that
\begin{equation}
\label{gen1}
\nnl V^T-V^T_N\nnr_0\leq 2\varepsilon.
\end{equation}
Put
\begin{equation*}
u_N=-\sqrt{\frac{\pi}{2}}\sum_{p=0}^N u_{l_p^N}^p\sum_{n=p}^N \omega_n h_p^n.
\end{equation*}
With regard to \eqref{gen1} and \eqref{ns}, we get
\begin{equation*}
\nnl W^T+\Phi_T u_N\nnr^0\leq 2\varepsilon. \qedhere
\end{equation*}
\end{proof}

\begin{remark}
The controls 
\begin{equation}
\label{contr}
u_N=-\sqrt{\frac{\pi}{2}}\sum_{p=0}^N u_{l_p^N}^p\sum_{n=p}^N \omega_n h_p^n,\quad N\in\N,
\end{equation}
found in the proof of Theorem \ref{thappr}
solve the approximate reachability problem for system \eqref{eq1}, \eqref{ic1}.
Here $u_{l_p^N}^p$ is defined by \eqref{contrl}, $h_p^n$ is defined by \eqref{hh} and $\omega_n$, $n=\overline{0,\infty}$, are the coefficients of decomposition of $W^T$ with respect to the basis  $\{\psi_n^T\}_{n=0}^\infty$.
\end{remark}

\begin{corollary}
Each state $W^T\in\THH^0$ is approximately reachable from any state $W^0\in\THH^0$ in a given time $T>0$.
\end{corollary}

\section{Controllability}\label{c}
\begin{definition}
For control system \eqref{eq1}, \eqref{ic1}, a state $W^0\in\THH^0$ is said to be null-controllable 
in a given time $T>0$ if $0\in\RR_T(W^0)$.
\end{definition}

In other words, a state $W^0\in\THH^0$ is  null-controllable 
in a given time $T>0$ iff there exists $u\in L^\infty(0,T)$ such that there exists a unique solution $W$ to \eqref{eq1}, \eqref{ic1}  and $W(\cdot,T)=0$.

\begin{theorem}
\label{thc}
If a state $W^0\in\THH^0$ is  null-controllable 
in a  time $T>0$, then $W^0=0$.
\end{theorem}

\begin{proof}
Find $u\in L^\infty(0,T)$ such that there exists a unique solution to \eqref{eq1}, \eqref{ic1}  and $W(\cdot,T)=0$. Denote $V^0=\FFF W^0$, $V(\cdot,t)=\FFF_{x\to\sigma}W(\cdot, t)$, $t\in[0,T]$.Taking into account \eqref{sol2}, we obtain
\begin{equation}
\label{ccc}
V^0(\sigma)=\sqrt{\frac2\pi}i\sigma\int_0^T e^{\xi\sigma^2}u(\xi)\, d\xi,\quad \sigma\in\R.
\end{equation}
Let $T^*>T$ be fixed. Then
\begin{equation*}
\sum_{m=0}^\infty\nu_m\frac{\widehat\psi_m^{T^*}}{\big(\big\|\widehat\psi_m^{T^*}\big\|\big)^2}
=\sum_{m=0}^\infty\int_0^T\mu_m(\xi)u(\xi)\, d\xi\frac{\widehat\psi_m^{T^*}}{\big(\big\|\widehat\psi_m^{T^*}\big\|\big)^2},
\end{equation*}
where
\begin{align}
\label{ccc1}
\nu_m&=2\int_0^\infty V^0(\sigma)\widehat\psi_m^{T^*}(\sigma)\,d\sigma,\\
\label{ccc2}
\mu_m(\xi)&=2i\sqrt{\frac2\pi}\int_0^\infty \sigma e^{\xi\sigma^2}\widehat\psi_m^{T^*}(\sigma)\, d\sigma.
\end{align}
Therefore,
\begin{equation}
\label{ccc3}
\int_0^T \mu_m(\xi)u(\xi)\,d\xi=\nu_m,\quad m=\overline{0,\infty}. 
\end{equation}
Let $m=\overline{0,\infty}$ be fixed. We have (see \eqref{hh})
\begin{align}
\label{ccc3a}
\mu_m(\xi)&=-2\sqrt{\frac2\pi}\sum_{p=0}^m h_p^m
\int_0^\infty \sigma^{2p+2} e^{-(T^*-\xi)\sigma^2}\,d\sigma\nonumber\\
&=(2m+1)!\frac{2\sqrt 2 T^*}{(T^*-\xi)^{3/2}}\sum_{p=0}^m\frac{(-1)^p}{(m-p)! p!}
\pl\frac{2T^*}{T^*-\xi}\pr^p\nonumber\\
&=(-1)^m \frac{(2m+1)!}{m!} \frac{2\sqrt2 T^*}{(T^*-\xi)^{3/2}}
\pl\frac{T^*+\xi}{T^*-\xi}\pr^m.
\end{align}
Replacing $\frac{T^*+\xi}{T^*-\xi}$ by $e^s$, we get
\begin{equation*}
\int_0^T \frac{T^*}{(T^*-\xi)^{3/2}}\pl\frac{T^*+\xi}{T^*-\xi}\pr^mu(\xi)\,d\xi
=\sqrt{\frac{T^*}{2}}\int_0^{\overline{T}} e^{ms} 
u\pl\frac{T^*(e^s-1)}{e^s+1}\pr\frac{e^s}{\sqrt{e^s+1}} \, ds,
\end{equation*}
where $\overline{T}=\ln\pl\frac{T^*+T}{T^*-T}\pr$. Denoting $U^*(s)=u\pl\frac{T^*(e^s-1)}{ e^s+1}\pr\frac{e^s}{\sqrt{e^s+1}}$, $s\in(0,\overline{T})$, 
$\nu_m^*=\frac{(-1)^m m!}{2\sqrt{T^*}(2m+1)!}\nu_m$, $m=\overline{0,\infty}$ and taking into account \eqref{ccc3}, \eqref{ccc3a}, we obtain
\begin{equation}
\label{ccc4}
\int_0^{\overline{T}} U^*(s) e^{ms}=\nu_m^*,\quad m=\overline{0,\infty}.
\end{equation}
Since
\begin{equation*}
|\nu_m|\le\nnl V^0\nnr_0\big\|\widehat \psi_m^{T^*}\big\|_0,\quad m=\overline{0,\infty},
\end{equation*}
taking into account \eqref{ortt} and the Stirling formula, we obtain
\begin{align*}
\nl \nu_m^*\nr\le\nnl V^0\nnr_0 \pl\frac{\pi}{T^*}\pr^{1/4}
\frac{2^{m-1/4}m!}{\sqrt{(2m+1)!}}
\sim \pl\frac{\pi^2}{2^3T^*}\pr^{1/4}
\frac{\nnl V^0\nnr_0}{(2m+1)^{1/4}}\quad \text{as } m\to\infty.
\end{align*}
Therefore, for all $\delta>0$ there exists $C_\delta>0$ such that
\begin{equation}
\label{ccc5}
\nl \nu_m^*\nr\le C_\delta e^{m\delta},\quad m=\overline{0,\infty}.
\end{equation}
We have
\begin{align}
\label{ccc6}
\int_0^{\overline{T}}\nl U^*(s)\nr^2\,ds
&=\int_0^T \nl u(\xi)\nr^2\frac{T^*+\xi}{(T^*-\xi)^2}\,d\xi
\le\pl\nnl u\nnr_{L^\infty(0,T)}\pr^2\int_0^T\frac{T^*+\xi}{(T^*-\xi)^2}\,d\xi\nonumber\\
&=\pl\nnl u\nnr_{L^\infty(0,T)}\pr^2 \pl \frac{2T}{T^*-T}-\ln\pl 1+\frac T{T^*-T}\pr\pr.
\end{align}
Thus,  $U^*\in L^2(0,T_*)$ and \eqref{ccc4}, \eqref{ccc5} hold. Due to \cite[Theorem~3.1, b)]{SMEZ}, we obtain $\nu_m^*=0$, $m=\overline{0,\infty}$, i.e., $V^0=W^0=0$.
\end{proof}

\section{Approximate controllability}\label{ac}

\begin{definition}
For control system  \eqref{eq1}, \eqref{ic1}, a state $W^0\in\THH^0$ is said to be approximately controllable to a target state $W^T\in\THH^0$
in a given time $T>0$ if $W^T\in\overline{\RR_T(W^0)}$, where the closure is considered in the space $\THH^0$. In particular, if $W^T=0$, the state $W^0$ is called approximately controllable.
\end{definition}

In other words, a state $W^0\in\THH^0$ is  approximately controllable  to a target state $W^T\in\THH^0$
in a given time $T>0$ iff for each $\varepsilon>0$ there exists $u_\varepsilon\in L^\infty(0,T)$ such that there exists a unique solution $W$ to \eqref{eq1}, \eqref{ic1} with $u=u_\varepsilon$ and $\nnl W(\cdot,T)-W^T \nnr^0<\varepsilon$.

Taking into account Theorem \ref{thappr}, one can see that the following  theorem holds.

\begin{theorem}
\label{thappc}
Each state $W^0\in\THH^0$ is  approximately controllable  to any target state $W^T\in\THH^0$
in a given time $T>0$.
\end{theorem}

\section{Examples}\label{ex}

The following two examples illustrates the results of Theorem \ref{thmomap}.
\begin{example}%
\label{ex1}
Let $T=1$, $W^T(x)=\sqrt{\frac{2}{\pi}}x
\int_0^T e^{-\frac{x^2}{4\xi}}\frac{d\xi}{2(2\xi)^{1/2}}$. Let us find controls $u_N(\xi)=v_N(T-\xi)$, $\xi\in[0,T]$, where $v_N$ is the solution to \eqref{mpf} for $N=2P-1$, $P\in\N$. We use the algorithm given in \cite{LGOR} to find $v_N$ in the form
\begin{equation}
\label{exx}
v_N(\xi)=\begin{cases}
1 & \text{if } \xi\in[\nu_{2p-1},\nu_{2p}],\ p=\overline{1,P},\\
0 & \text{if } \xi\in[\nu_{2p},\nu_{2p+1}],\ p=\overline{0,P},
\end{cases}
\end{equation}
where $0=\nu_0\le\nu_1\le\nu_2\le\nu_3\le\dots\le\nu_{2P-1}\le\nu_{2P}\le\nu_{2P+1}=T$. By $W_N$ we denote the value at $t=T$ of the solution to \eqref{eq1}, \eqref{ic1} with the control $u=u_N$. Influence of controls $u_N$, $N=3,5,7,15$, on the end states of solutions $W_N$ is given in Figure \ref{fig:mom2}.
\begin{figure}[!h]
\begin{center}
\begin{subfigure}[b]{0.48\linewidth}
\centering \includegraphics[width=\textwidth]{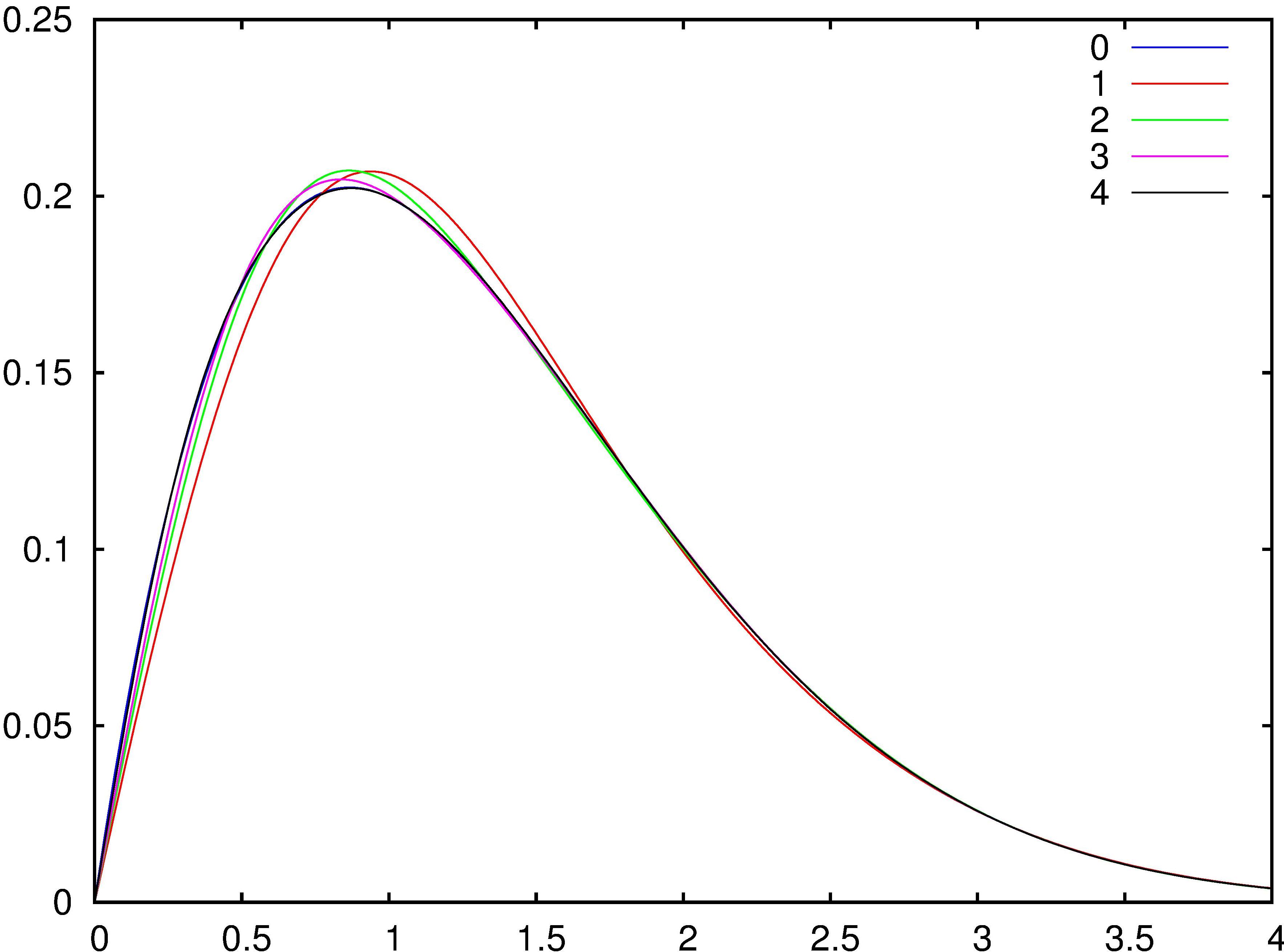}\\
\centering \caption{\parbox[t]{0.92\textwidth}{%
Influence of the control $u_N$ on the end state $W_N$ in the cases: 
\zcx{0}~$u=0$, \zcx{1}~$N=3$, \zcx{2}~$N=5$, \zcx{3}~$N=7$, \zcx{4}~$N=15$.}}
\label{fig:mom22}
\end{subfigure}
\begin{subfigure}[b]{0.49\linewidth}
\centering \includegraphics[width=\textwidth]{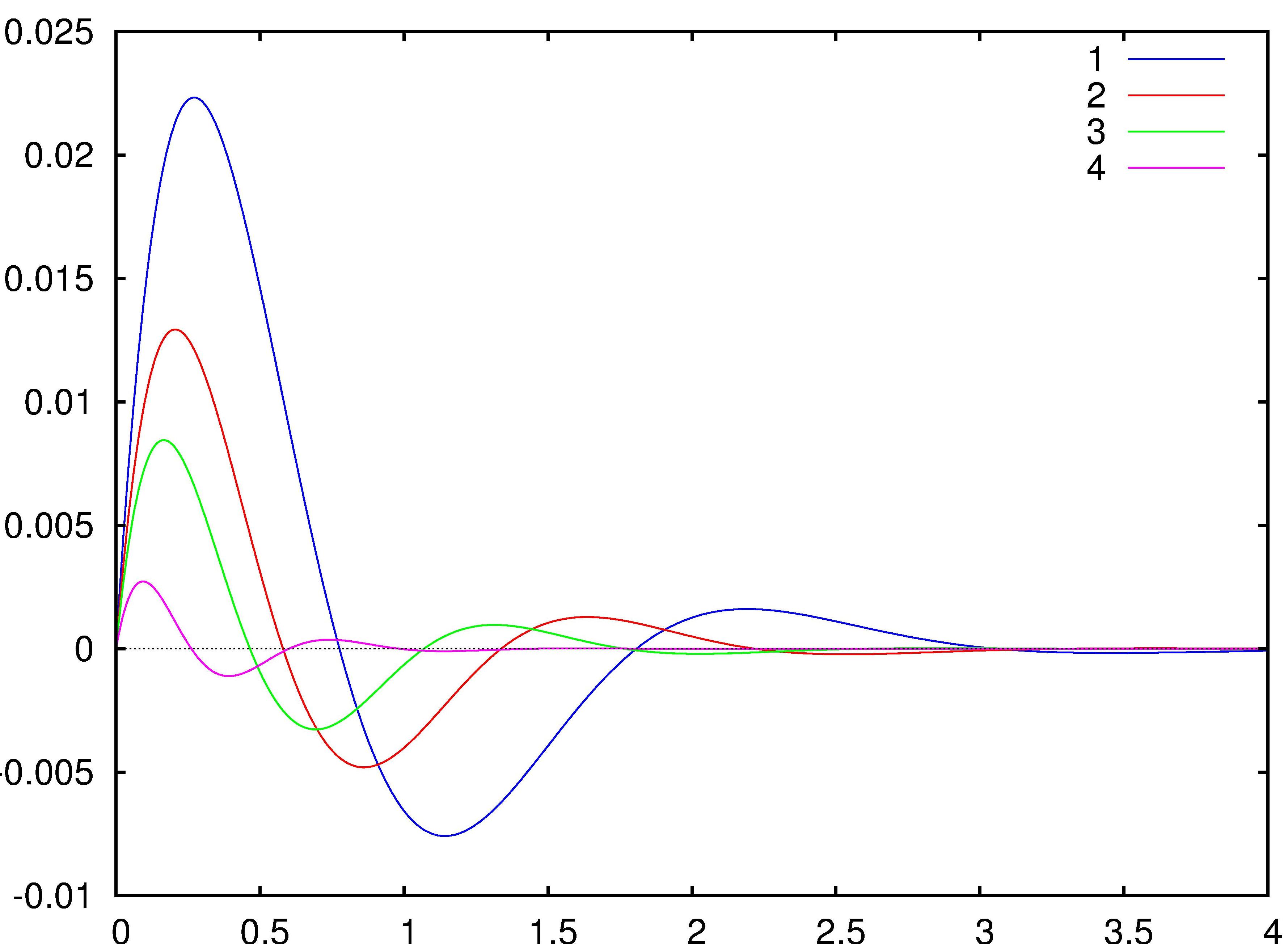}
\centering \caption{\parbox[t]{0.92\textwidth}{%
The differences $W^T-W_N$ in the cases:
\zcx{1}~$N=3$, \zcx{2}~$N=5$, \zcx{3}~$N=7$, \zcx{4}~$N=15$.\\ \mbox{}}}
\label{fig:mom21}
\end{subfigure}
      \caption{Influence of the control $u_N$ on the end state of solution to \eqref{eq1}, \eqref{ic1} with the control $u=u_N$ and the target state  $W^T(x)=\sqrt{\frac{2}{\pi}}x
\int_0^T e^{-\frac{x^2}{4\xi}}\frac{d\xi}{2(2\xi)^{1/2}}$.}\label{fig:mom2}
\end{center}
\end{figure}
\end{example}

\begin{example}%
\label{ex2}
Let $T=1$, $W^T(x)=\sqrt{\frac{2}{\pi}}x
\int_0^T e^{-\frac{x^2}{4\xi}}\frac{1-\xi}{(2\xi)^{3/2}}d\xi$. Let us find controls $u_N(\xi)=v_N(T-\xi)$, $\xi\in[0,T]$, where $v_N$ is the solution to \eqref{mpf} for $N=2P-1$, $P\in\N$. We use the algorithm given in \cite{LGOR} to find $v_N$ in the form \eqref{exx}. By $W_N$ we denote the value at $t=T$ of the solution to \eqref{eq1}, \eqref{ic1} with the control $u=u_N$. Influence of controls $u_N$, $N=3,5,7,15$, on the end states of solutions $W_N$ is given in Figure \ref{fig:mom3}.
\begin{figure}[!h]
\begin{center}
\begin{subfigure}[b]{0.48\linewidth}
\centering \includegraphics[width=\textwidth]{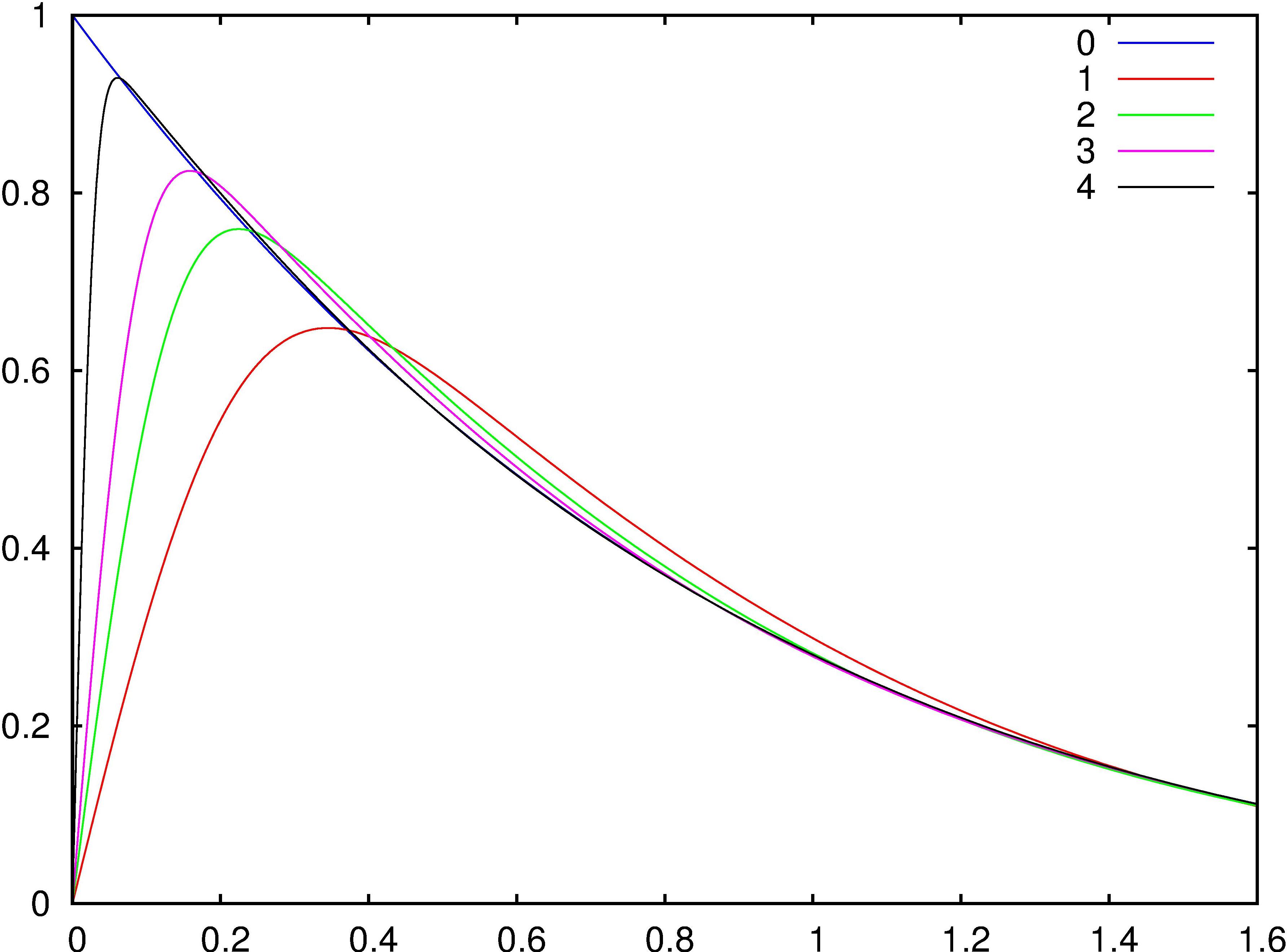}\\
\centering \caption{\parbox[t]{0.92\textwidth}{Influence of the control $u_N$ on the end state $W_N$ in the cases: 
\zcx{0}~$u=0$, \zcx{1}~$N=3$, \zcx{2}~$N=5$, \zcx{3}~$N=7$, \zcx{4}~$N=15$.}}
\label{fig:mom32}
\end{subfigure}
\begin{subfigure}[b]{0.49\linewidth}
\centering \includegraphics[width=0.98\textwidth]{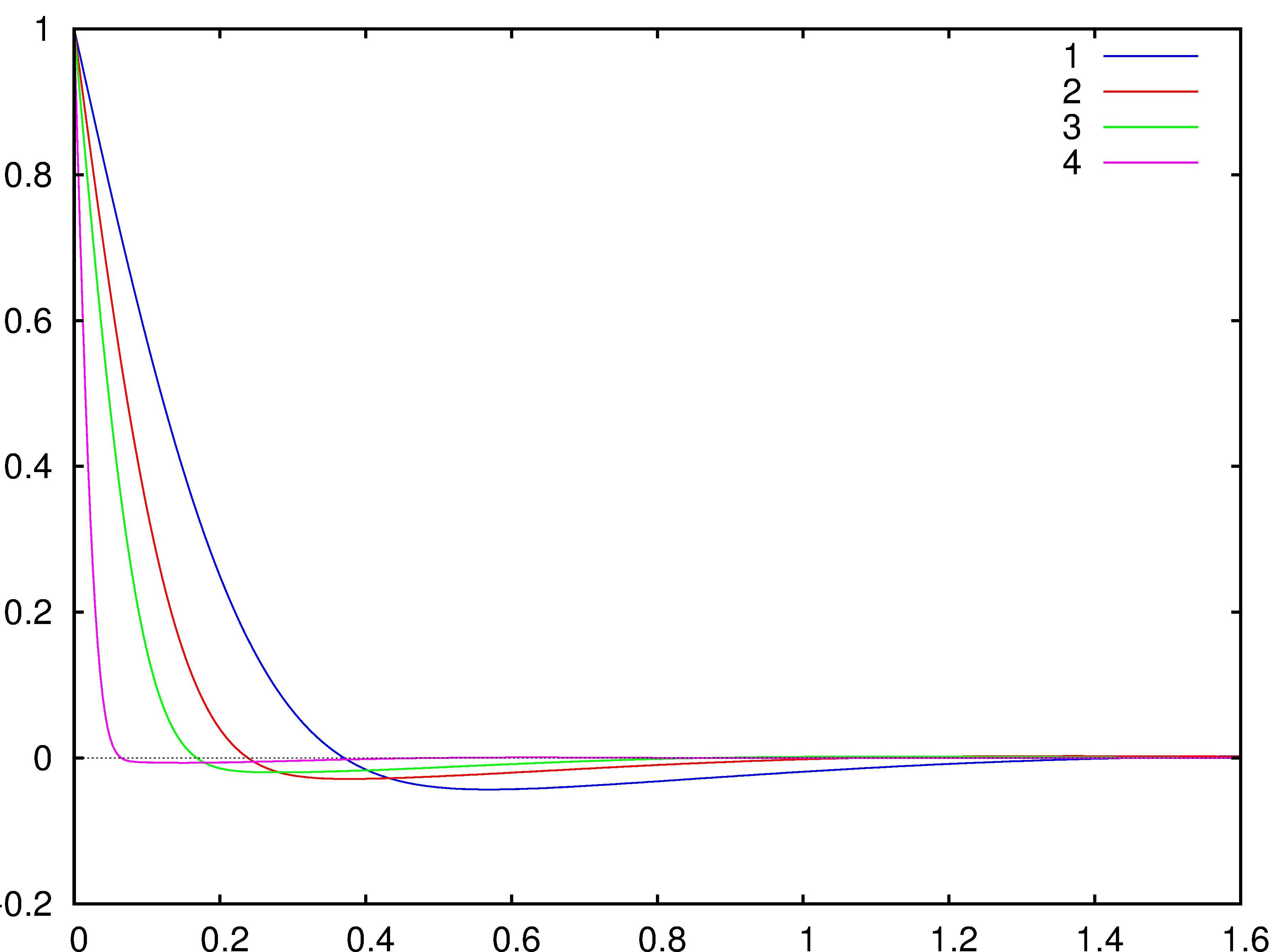}
\centering \caption{\parbox[t]{0.92\textwidth}{The differences $W^T-W_N$ in the cases:
\zcx{1}~$N=3$, \zcx{2}~$N=5$, \zcx{3}~$N=7$, \zcx{4}~$N=15$.\\ \mbox{}}}
\label{fig:mom31}
\end{subfigure}
      \caption{Influence of the control $u_N$ on the end state of solution to \eqref{eq1}, \eqref{ic1} with the control $u=u_N$ and the target state $W^T(x)=\sqrt{\frac{2}{\pi}}x
\int_0^T e^{-\frac{x^2}{4\xi}}\frac{1-\xi}{(2\xi)^{3/2}}d\xi$.}\label{fig:mom3}
\end{center}
\end{figure}
\end{example}

The following example illustrates the result of Theorem \ref{thappr}.

\begin{example}
\label{ex3}
Let $W^T(x)=2\sqrt{\frac{2}{\pi}}e^{\frac14}e^{-\frac{x^2}{4T}}\sin\frac{x}{\sqrt{2T}}$.
Consider the reachability problem for system \eqref{eq1}, \eqref{ic1} with $W^0=0$.
Denote $V^T=\FFF W^T$. Then $V^T(\sigma)=-4i\sqrt{\frac{T}{\pi}}e^{-\frac14}e^{-T\sigma^2}\sinh\sqrt{2T}\sigma$.
Since $V^T=\sum_{n=0}^\infty\omega_n\widehat{\psi}_n^T$, then it is easy to see that
$V^T(\sigma)=ie^{-T\sigma^2}\sum_{p=0}^\infty\sigma^{2p+1}\sum_{n=p}^\infty\omega_n h_p^n,$
where $h_p^n$ is defined by \eqref{hh} and $\omega_n=\sqrt{\frac{2}{\pi}}\frac{(-1)^n}{2^{2n}(2n+1)!}$.

For each $N\in\N$, denote $g_p^N=\sum_{n=p}^N \omega_n h_p^n$. Denote also
\begin{align*}
V_N(\sigma)&=i\sum_{p=0}^N g_p^N\varphi_p(\sigma)=ie^{-T\sigma^2}\sum_{p=0}^N g_p^N\sigma^{2p+1},\\
V_N^l(\sigma)&=i\sum_{p=0}^N g_p^N\varphi_p^l(\sigma)=ie^{-T\sigma^2}\sum_{p=0}^N g_p^N\sigma^{2p+1}\pl\frac{e^{\sigma^2/l}-1}{\sigma^2/l}\pr^{p+1}.
\end{align*}
Then,
\begin{equation}
\label{estim}
\nnl V^T-V_N^l\nnr_0\leq\nnl V^T-V_N\nnr_0+\nnl V_N^l-V_N\nnr_0.
\end{equation}
Using \eqref{ortt}, we get
\begin{align}
\nnl V^T-V_N\nnr_0&=\sqrt{\frac{2}{\pi}}\pl\sum_{n=N+1}^\infty\pl\frac{(-1)^n}{2^{2n}(2n+1)!}\pr^2\sqrt{2\pi T}2^{2n+1}(2n+1)!\pr^\frac{1}{2}\nonumber\\
&\leq\sqrt{8}\pl\frac{2T}{\pi}\pr^\frac{1}{4}\sqrt{\frac{\cosh\frac{1}{2}}{2^{2N+3}(2N+3)!}}.\label{estim1}
\end{align}
We have
\begin{align}
\label{estim2}
\nnl V_N^l-V_N\nnr_0\leq\sum_{p=0}^N \nl g_p^N\nr\nnl\varphi_p^l-\varphi_p\nnr_0.
\end{align}
Substituting $h_p^n$ and $\omega_n$ in $g_p^N$, we obtain
\begin{align}
\nl g_p^N\nr&=\sqrt{\frac{2}{\pi}}\nl\sum_{n=p}^N\frac{(-1)^{n+p+1}2^{2p+1}(2T)^{p+1}(2n+1)!}{2^{2n}(2n+1)!(n-p)!(2p+1)!}\nr\nonumber\\
&=2\sqrt{\frac{2}{\pi}}\frac{(2T)^{p+1}}{(2p+1)!}\nl\sum_{n=p}^N\frac{(-1)^{n-p}}{2^{2(n-p)}(n-p)!}\nr\leq 2\sqrt{\frac{2}{\pi}}\frac{(2T)^{p+1}}{(2p+1)!}e^{-\frac{1}{4}}.\label{estim3}
\end{align}
Evidently, the following three estimates hold:
\begin{align*}
\nl (y+1)^{p+1}-1\nr\leq(p+1)(y+1)^p y,\qquad y>0,\\
\frac{e^z-1}{z}\leq e^z,\quad \frac{e^z-1}{z}-1\leq\frac{1}{2}ze^z,\qquad z>0.
\end{align*}
Therefore,
\begin{align*}
\nl\pl\frac{e^{\sigma^2/l}-1}{\sigma^2/l}\pr^{p+1}-1\nr&\leq(p+1)\pl\frac{e^{\sigma^2/l}-1}{\sigma^2/l}\pr^p\pl\frac{e^{\sigma^2/l}-1}{\sigma^2/l}-1\pr\\
&\leq\frac{p+1}{2l}\sigma^2e^{(p+1)\sigma^2/l}.
\end{align*}
From here, it follows that 
\begin{align}
\nnl\varphi_p^l-\varphi_p\nnr_0&=\pl 2\int_0^\infty\pl\sigma^{2p+1}e^{-T\sigma^2}\nl\pl\frac{e^{\sigma^2/l}-1}{\sigma^2/l}\pr^{p+1}-1\nr\pr^2 d\sigma\pr^\frac{1}{2}\nonumber\\
&\leq\pl\frac{(p+1)^2}{2l^2}\int_0^\infty\pl\sigma^{2p+3}e^{-\sigma^2(T-(p+1)/l)}\pr^2 d\sigma\pr^\frac{1}{2}\nonumber\\
&\leq\pl\frac{(p+1)^2}{2l^2}\int_0^\infty\pl\sigma^{2p+3}e^{-\frac{3}{4}T\sigma^2}\pr^2 d\sigma\pr^\frac{1}{2},\label{estim4}
\end{align}
if $\frac{p+1}{l}<\frac{T}{4}$. Since $\max_{\sigma>0}\sigma^{2p+3}e^{-T\sigma^2/2}=\pl\frac{2p+3}{T}\pr^{p+3/2}e^{-(2p+3)/2}$, then we get
\begin{align*}
\nnl\varphi_p^l-\varphi_p\nnr_0&\leq\pl\frac{(p+1)^2}{2l^2}\pl\frac{2p+3}{T}\pr^{2p+3}e^{-(2p+3)}\int_0^\infty e^{-\frac{T\sigma^2}{2}}d\sigma\pr^\frac{1}{2}\\
&\leq\pl\frac{2\pi}{T}\pr^\frac{1}{4}\frac{p+1}{l}\frac{2^{p+1/2}}{T^{p+3/2}}\pl\frac{p+2}{e}\pr^{p+2}.
\end{align*}
From  here, using the Stirling formula \eqref{stirl}, we obtain 
\begin{align}
\nnl\varphi_p^l-\varphi_p\nnr_0\leq\pl\frac{1}{2\pi T}\pr^\frac{1}{4}\frac{\sqrt{p+2}}{l}\frac{2^{p+1/2}}{T^{p+3/2}}(p+2)!.\label{estim5}
\end{align}
According to \eqref{estim3} and \eqref{estim5} and continuing \eqref{estim2}, we have
\begin{align}
\label{estim6}
\nnl V_N^l-V_N\nnr_0&\leq\sum_{p=0}^N
2\sqrt{\frac{2}{\pi}}\frac{(2T)^{p+1}}{(2p+1)!}e^{-\frac{1}{4}}\pl\frac{1}{2\pi
T}\pr^\frac{1}{4}\frac{\sqrt{p+2}}{l}\frac{2^{p+1/2}}{T^{p+3/2}}(p+2)!\nonumber\\
&=\frac{2^\frac{11}{4}}{l}\pl\frac{1}{T^3\pi^3e}\pr^\frac{1}{4}\sum_{p=0}^N\frac{2^{2p}\sqrt{p+2}(p+2)!}{(2p+1)!}.
\end{align}
From \eqref{estim}, taking into account \eqref{estim1} and \eqref{estim6}, we get 
\begin{align}
\nnl V^T-V_N^l\nnr_0&\leq \sqrt{8}\pl\frac{2T}{\pi}\pr^\frac{1}{4}\sqrt{\frac{\cosh\frac{1}{2}}{2^{2N+3}(2N+3)!}}\nonumber\\
&+\frac{2^\frac{11}{4}}{l}\pl\frac{1}{T^3\pi^3e}\pr^\frac{1}{4}\sum_{p=0}^N\frac{2^{2p}\sqrt{p+2}(p+2)!}{(2p+1)!}.\label{estim7}
\end{align}
For the last sum, we have
\begin{align*}
\sum_{p=0}^N\frac{2^{2p}\sqrt{p+2}(p+2)!}{(2p+1)!}\leq\sum_{p=0}^N\frac{(p+1)(p+2)^{3/2}}{p!}\leq 26+8e.
\end{align*}
Therefore, \eqref{estim7} takes the form
\begin{align*}
\nnl V^T-V_N^l\nnr_0\leq \sqrt{8}\pl\frac{2T}{\pi}\pr^\frac{1}{4}\sqrt{\frac{\cosh\frac{1}{2}}{2^{2N+3}(2N+3)!}}
+2^\frac{11}{4}\pl\frac{1}{T^3\pi^3e}\pr^\frac{1}{4}\frac{1}{l}(26+8e).
\end{align*}

Due to Theorem \eqref{thappr}, we obtain $W_N^l=-\Phi_Tu_N$. With regard to \eqref{contr}, we get
\begin{equation*}
W_N^l(x)=-x\int_0^T\frac{u_N^l(\xi)}{\pl 2(T-\xi)\pr^{3/2}}e^{-\frac{x^2}{4(T-\xi)}}d\xi,
\end{equation*}
where $u_N^l=\sum_{p=0}^Ng_p^N u_l^p$. Some estimates for $\nnl W^T-W_N^l\nnr^0$ are given in the Table \ref{tab1} and influence of the control $u_N^l$ on the end state $W_N^l$ of solution to \eqref{eq1}, \eqref{ic1} with the control $u=u_N^l$ and the target state $W^T$ is shown in Figure \ref{fig:appr}.

\begin{table}[!h]
\centering
\begin{tabular}{|l|l|l|l|}
\hline
&$\varepsilon_1$&$\varepsilon_2$&$\varepsilon$\\
\hline
$N=1$, $l=10$& 0.0433 & 2.1662 & 2.2095\\
$N=1$, $l=100$& 0.0433 & 0.2167 & 0.2600\\
$N=2$, $l=100$& 0.0034 & 0.3588 & 0.3622\\
$N=2$, $l=1000$& 0.0034 & 0.0359 & 0.0393\\
\hline
\end{tabular}
\caption{Estimates for $\nnl W^T-W_N^l\nnr^0$, $\varepsilon_1=\sqrt{8}\pl\frac{2T}{\pi}\pr^\frac{1}{4}\sqrt{\frac{\cosh\frac{1}{2}}{2^{2N+3}(2N+3)!}}$, $\varepsilon_2=2^\frac{11}{4}\pl\frac{1}{T^3\pi^3e}\pr^\frac{1}{4}\frac{1}{l}\sum_{p=0}^N\frac{2^{2p}\sqrt{p+2}(p+2)!}{(2p+1)!}$,  $\varepsilon=\varepsilon_1+\varepsilon_2$ (see \eqref{estim7}).}\label{tab1}
\end{table}

\begin{figure}[!h]
\begin{center}
\begin{subfigure}[b]{0.48\linewidth}
\centering \includegraphics[width=\textwidth]{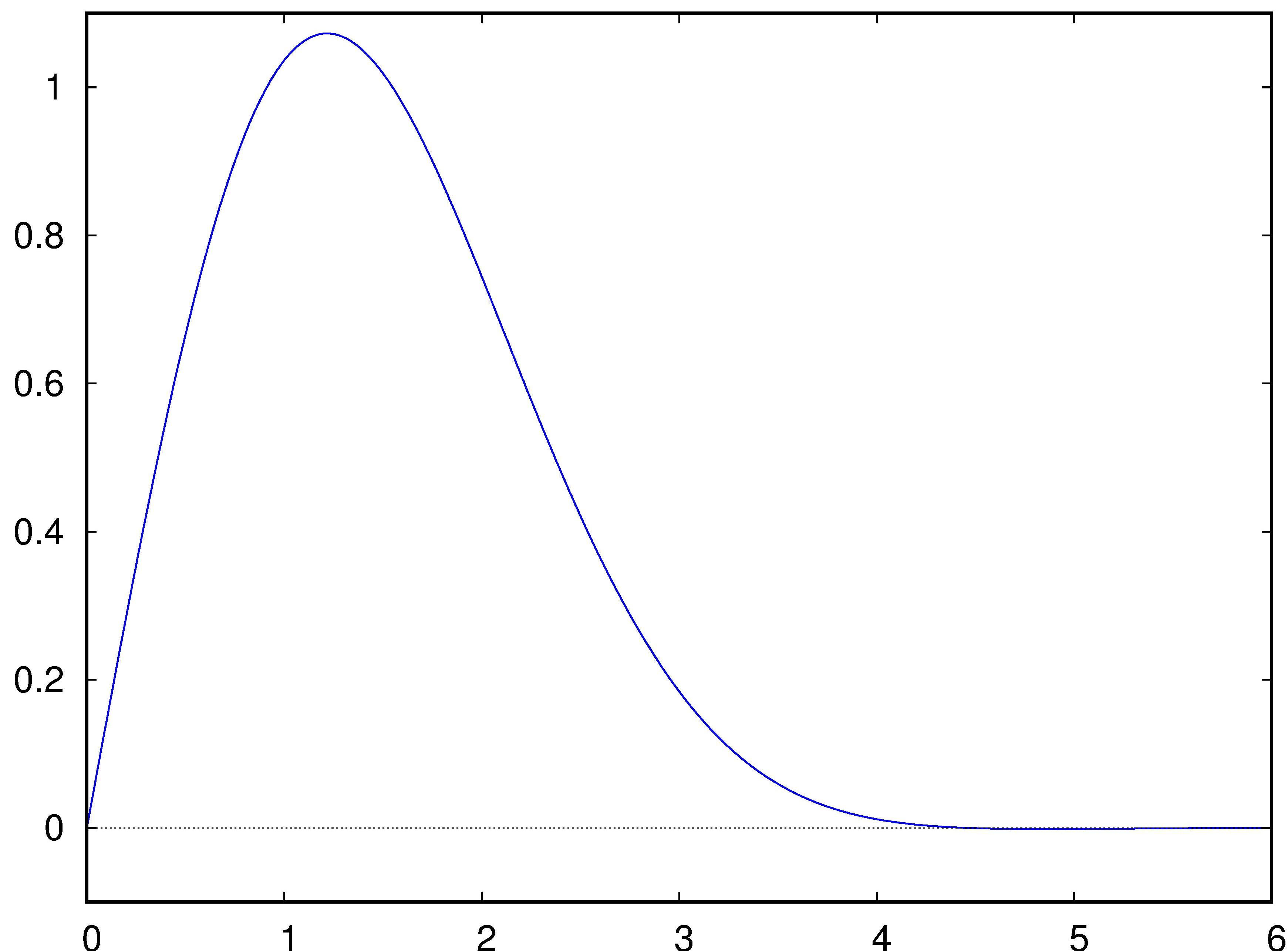}\\
\centering \caption{\parbox[t]{0.92\textwidth}{The given $W^T(x)$.\\[0.19\baselineskip]  \mbox{}\\ \mbox{}}}
\label{fig:appr12}
\end{subfigure}
\begin{subfigure}[b]{0.49\linewidth}
\centering \includegraphics[width=\textwidth]{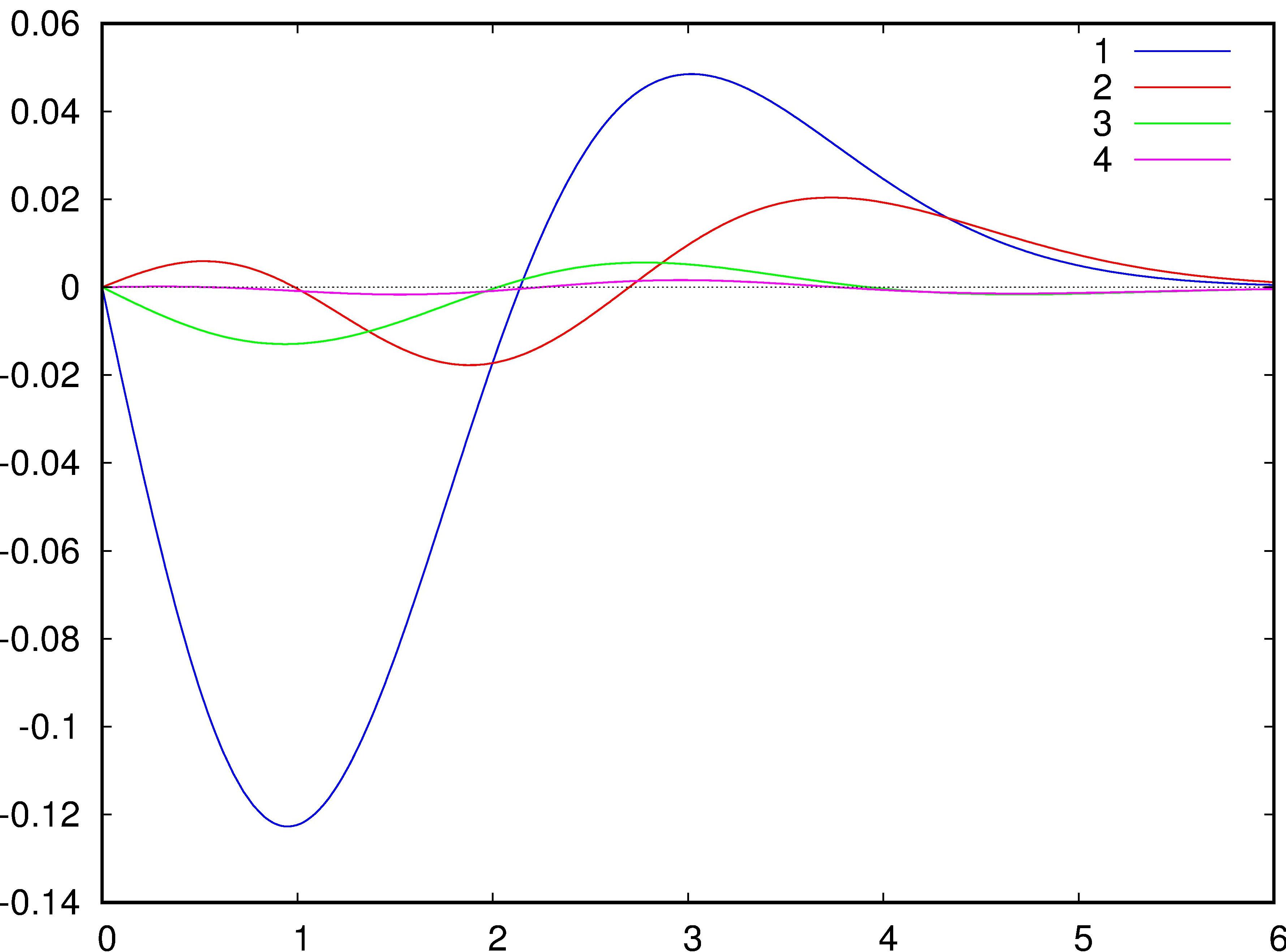}
\centering \caption{\parbox[t]{0.92\textwidth}{%
The differences $W^T-W_N^l$ in the cases: \zcx{1}~$N=1$, $l=10$; \zcx{2}~$N=1$, $l=100$; \zcx{3}~$N=2$, $l=100$; \zcx{4}~$N=2$, $l=1000$.}}
\label{fig:appr1}
\end{subfigure}
     \centering \caption{Influence of the control $u_N^l$ on the end state $W_N^l$ of solution to \eqref{eq1}, \eqref{ic1} with the control $u=u_N^l$ and the target state $W^T(x)=\frac4{\sqrt{{2}{\pi}}}e^{\frac14}e^{-\frac{x^2}{4T}}\sin\frac{x}{\sqrt{2T}}$.}\label{fig:appr}
\end{center}
\end{figure}

\end{example}



\vskip 0.5em
\par\noindent\small\textrm{\upshape Larissa Fardigola,}
\par%
\vskip 0.1em
\par\noindent\parbox[t]{\textwidth}{\small\textrm{\slshape  						B. Verkin Institute for Low Temperature Physics and Engineering of the
National Academy of Sciences of Ukraine, 47 Nauky Ave., Kharkiv, 61103, Ukraine,}}
\par%
\vskip 0.1em
\par\noindent\small{E-mail:\ \href{mailto:fardigola@ilt.kharkov.ua}{{\mdseries\ttfamily fardigola@ilt.kharkov.ua}}}
\par
\vskip 0.1em
\vskip 0.5em
\par\noindent\small\textrm{\upshape Kateryna  Khalina,}
\par%
\vskip 0.1em
\par\noindent\parbox[t]{\textwidth}{\small\textrm{\slshape  						B. Verkin Institute for Low Temperature Physics and Engineering of the
National Academy of Sciences of Ukraine, 47 Nauky Ave., Kharkiv, 61103, Ukraine,}}
\par%
\vskip 0.1em
\par\noindent\small{E-mail:\ \href{mailto:fardigola@ilt.kharkov.ua}{{\mdseries\ttfamily khalina@meta.ua}}}
\par
\vskip 0.1em


\end{document}